\documentclass{article}

\usepackage{amsmath,amssymb,amsthm,mathrsfs}

\usepackage{cite}

\usepackage{a4wide}

\newtheorem{theorem}{Theorem}
\newtheorem{lemma}{Lemma}
\newtheorem{corollary}{Corollary}

\theoremstyle{definition}
\newtheorem{definition}{Definition}

\theoremstyle{remark}
\newtheorem{remark}{Remark}

\begin{document}

\author{Amar Debbouche$^{a}$\\
\texttt{amar$_{-}$debbouche@yahoo.fr}
\and
Delfim F. M. Torres$^{b}$\\
\texttt{delfim@ua.pt}}

\title{Sobolev Type Fractional Dynamic Equations and Optimal Multi-Integral Controls 
with Fractional Nonlocal Conditions\thanks{This is a preprint of a paper whose final 
and definite form will appear in \emph{Fract. Calc. Appl. Anal.} 
Paper submitted 10/April/2014; accepted for publication 21/Sept/2014.}}

\date{$^{a}$Department of Mathematics, Guelma University,
24000 Guelma, Algeria\\[0.3cm]
$^{b}$\text{Center for Research and Development in Mathematics and Applications (CIDMA)}\\
Department of Mathematics, University of Aveiro, 3810--193 Aveiro, Portugal}

\maketitle


\begin{abstract}
We prove existence and uniqueness of mild solutions
to Sobolev type fractional nonlocal
dynamic equations in Banach spaces. The Sobolev
nonlocal condition is considered in terms of a
Riemann--Liouville fractional derivative.
A Lagrange optimal control problem is considered,
and existence of a multi-integral solution obtained.
Main tools include fractional calculus,
semigroup theory, fractional power of operators,
a singular version of Gronwall's inequality,
and Leray--Schauder fixed point theorem.
An example illustrating the theory is given.

\bigskip


\medskip

\end{abstract}


\section{Introduction}

Fractional differential equations have attracted the attention of scientists,
in reason to their accurate, helpful, and successful results in fields such
as mathematical modelling of physical, engineering, and biological phenomena.
Both theoretical and practical aspects of the subject are being explored.
In particular, fractional differential equations provide an excellent tool to describe hereditary properties
of various materials and processes, finding numerous applications in viscoelasticity,
electrochemistry, porous media, and electromagnetism.
The reader interested in the development of the theory, methods, and applications
of fractional calculus is referred to the books
\cite{AMA.1,AMA.2,AMA.3,AMA.4,AMA.5,AMA.6,AMA.7,AMA.8,AMA.9}
and to the papers \cite{AMA.10,AMA.11,AMA.12,AMA.13,AMA.14,AMA.15,AMA.16,AMA.17}.
For recent developments in the area of nonlocal fractional differential equations and inclusions
see \cite{MR2897776,AMA.18,AMA.19,AMA.20,AMA.21,AMA.22,AMA.23} and references therein.

The study of fractional control systems and fractional optimal control problems
is under intense investigation \cite{AMA.24,AMA.25,AMA.26}.
Those control systems are most often based on the principle of feedback,
whereby the signal to be controlled is compared to a desired reference
signal and the discrepancy used to compute corrective control actions \cite{AMA.27}.
The fractional optimal control of a distributed system is an optimal control problem
for which the system dynamics is defined by means of fractional differential equations \cite{AMA.28}.
In our previous work \cite{AMA.21}, we introduced multi-delay controls
and we investigated a nonlocal condition for fractional semilinear control systems.
The existence of optimal pairs for systems governed by fractional evolution equations with initial
and nonlocal conditions, is also presented by Wang et al. \cite{AMA.23} and Wang and Zhou \cite{AMA.29}.
Here we are concerned with the study of fractional nonlinear evolution equations
subject to fractional Sobolev nonlocal conditions. Sobolev type semilinear equations
serve as an abstract formulation of partial differential equations,
which arise in various applications, such as in the flow of fluid through fissured rocks, thermodynamics,
and shear in second order fluids. Moreover, fractional differential equations of Sobolev type appear in the
theory of control of dynamical systems, when the controlled system and/or the controller is described
by a fractional differential equation of Sobolev type \cite{MR3181063}.
The mathematical modeling and simulations of such systems and processes
are based on the description of their properties in terms of fractional differential equations of Sobolev type. These
new models are claimed to be more adequate than previously used integer order models, so fractional order differential equations
of Sobolev type have been investigated by many researchers, e.g., in \cite{AMA.30,MR3154621,MR3129325,AMA.31}.
Motivated by these facts, we introduce here a new nonlocal fractional condition of Sobolev type
and we present the optimal control of multiply integrated Sobolev type nonlinear
fractional evolution equations. The problem requires to formulate a new solution operator
and its properties, such as boundedness and compactness. Further, we present a class of admissible
multi-integral controls and we prove, under an appropriate set of sufficient conditions,
an existence result of optimal multi-integral controls for a Lagrange optimal control problem,
denoted in the sequel by \eqref{LP}. More precisely, we are concerned with the study
of fractional nonlinear evolution equations
\begin{equation}
\label{eq:1.1}
^CD^{\alpha}_{t}[Lu(t)]=Eu(t)+f(t, W(t))
\end{equation}
subject to fractional Sobolev nonlocal conditions
\begin{equation}
\label{eq:1.2}
^{L}D^{1-\alpha}_{t}[Mu(0)]=u_{0}+h(u(t)),
\end{equation}
where $^CD^{\alpha}_{t}$ and $^{L}D^{1-\alpha}_{t}$ are, respectively,
Caputo and Riemann--Liouville fractional derivatives
with $0<\alpha\leq1$ and $t\in J=[0, a]$. Let $X$ and $Y$ be two Banach spaces such that $Y$ is
densely and continuously embedded in $X$, the unknown function $u(\cdot)$ takes its values in $X$ and $u_{0}\in X$.
We consider the operators $L: D(L)\subset X\rightarrow Y$, $E: D(E)\subset X\rightarrow Y$ and
$M: D(M)\subset X\rightarrow X$, $W(t)=(B_{1}(t)u(t),\ldots,B_{r}(t)u(t))$, such that
$\lbrace B_{i}(t):i=1,\ldots,r,~t\in J\rbrace$ is a family of linear closed operators defined on dense sets
$S_{1},\ldots,S_{r}$ in $X$ with values in $Y$. It is also assumed that $f: J\times X^{r}\rightarrow Y$
and $h: C(J: X)\rightarrow X$ are given abstract functions satisfying some conditions to be specified later.
In Section~\ref{sec:2} we present some essential notions and facts
that will be used in the proof of our results, such as, fractional operators, fractional powers
of the generator of an analytic compact semigroup, and the form of mild solutions
of \eqref{eq:1.1}--\eqref{eq:1.2}. In Section~\ref{sec:3}, we prove existence (Theorem~\ref{Theorem 3.1})
and uniqueness (Theorem~\ref{Theorem 3.2}) of mild solutions to system \eqref{eq:1.1}--\eqref{eq:1.2}.
Then, in Section~\ref{sec:4}, we prove existence of optimal pairs for the \eqref{LP} Lagrange optimal control problem
(Theorem~\ref{Theorem 4.2}). We end with Section~\ref{sec:5}, where an example illustrating the application
of the abstract results (Theorems~\ref{Theorem 3.1}, \ref{Theorem 3.2} and \ref{Theorem 4.2}) is given.


\section{Preliminaries}
\label{sec:2}

In this section we introduce some basic definitions, notations and lemmas,
which will be used throughout the work. In particular, we give main properties of fractional calculus
\cite{AMA.3,AMA.4} and well known facts in semigroup theory \cite{AMA.32,AMA.33,AMA.34}.

\begin{definition}
\label{Definition 2.1}
The fractional integral of order $\alpha>0$ of a function $f\in L^{1}([a,b],\mathbb{R}^{+})$ is given by
$$
I^{\alpha}_{a}f(t)=\frac{1}{\Gamma(\alpha)}\int_{a}^{t}(t-s)^{\alpha-1}f(s)ds,
$$
where $\Gamma$ is the classical gamma function. If $a=0$, we can write $I^{\alpha}f(t)=(g_{\alpha}*f)(t)$, where
$$
g_{\alpha}(t):=\left\{
\begin{array}{ll}
\frac{1}{\Gamma(\alpha)}t^{\alpha-1},& \mbox{$t>0$},\\
0, & \mbox{$t\leq 0$},
\end{array}\right.
$$
and, as usual, $*$ denotes convolution of functions. Moreover,
$\lim\limits_{\alpha\rightarrow 0}g_{\alpha}(t)=\delta(t)$,
with $\delta$ the delta Dirac function.
\end{definition}

\begin{definition}
\label{Definition 2.2}
The Riemann--Liouville fractional derivative of order $\alpha>0$,
$n-1<\alpha<n$, $n\in \mathbb{N}$, is given by
$$
^{L}D^{\alpha}f(t)=\frac{1}{\Gamma(n-\alpha)}\frac{d^{n}}{dt^{n}}
\int_{0}^{t}\frac{f(s)}{(t-s)^{\alpha+1-n}}ds,\quad t>0,
$$
where function $f$ has absolutely continuous derivatives up to order $n-1$.
\end{definition}

\begin{definition}
\label{Definition 2.3}
The Caputo fractional derivative of order $\alpha>0$,
$n-1<\alpha<n$, $n\in \mathbb{N}$, is given by
$$
^{C}D^{\alpha}f(t) = {^{L}D}^{\alpha}\left(f(t)
-\sum\limits_{k=0}^{n-1}\frac{t^{k}}{k!}f^{(k)}(0)\right),
\quad t>0,
$$
where function $f$ has absolutely continuous derivatives up to order $n-1$.
\end{definition}

\begin{remark}
\label{Remark 2.1}
Let $n-1<\alpha<n$, $n\in \mathbb{N}$. The following properties hold:
\begin{itemize}
\item[(i)] If $f\in C^{n}([0, \infty))$, then
$$
^{C}D^{\alpha}f(t)=\frac{1}{\Gamma(n-\alpha)}
\int_{0}^{t}\frac{f^{(n)}(s)}{(t-s)^{\alpha+1-n}}ds=I^{n-\alpha}f^{(n)}(t),
\quad t>0.
$$
\item[(ii)] The Caputo derivative of a constant function is equal to zero.
\item[(iii)] The Riemann--Liouville derivative of a constant function $C$ is given by
$$
^{L}D^{\alpha}_{a^{+}}C=\frac{C}{\Gamma(1-\alpha)}(x-a)^{-\alpha}.
$$
\end{itemize}
If $f$ is an abstract function with values in $X$, then the integrals
which appear in Definitions~\ref{Definition 2.1}--\ref{Definition 2.3}
are taken in Bochner's sense.
\end{remark}

Let $(X, \Vert\cdot\Vert)$ be a Banach space, $C(J, X)$ denotes the Banach space
of continuous functions from $J$ into $X$ with the norm
$\Vert u\Vert_{J}=\sup \lbrace\Vert u(t)\Vert: t\in J\rbrace$,
and let $\mathcal{L}(X)$ be the Banach space of bounded linear operators
from $X$ to $X$ with the norm
$\Vert G\Vert_{\mathcal{L}(X)}=\sup\lbrace\Vert G(u)\Vert: \Vert u\Vert=1\rbrace$.
We make the following assumptions:
\begin{itemize}
\item[($H_1$)] $E: D(E)\subset X\rightarrow Y$ is linear, closed, and
$L: D(L)\subset X\rightarrow Y$ and $M: D(M)\subset X\rightarrow X$ are linear operators.
\item[($H_2$)] $D(M)\subset D(L)\subset D(E)$ and $L$ and $M$ are bijective.
\item[($H_3$)] $L^{-1}: Y\rightarrow D(L)\subset X$ and
$M^{-1}: X\rightarrow D(M)\subset X$ are linear, bounded, and compact operators.
\end{itemize}
Note that ($H_3$) implies $L$ to be closed. Indeed, if $L^{-1}$ is closed and injective, then
its inverse is also closed. From ($H_1$)--($H_3$) and the closed graph theorem,
we obtain the boundedness of the linear operator $EL^{-1}: Y\rightarrow Y$. Consequently,
$EL^{-1}$ generates a semigroup $\lbrace Q(t), t\geq0\rbrace,~ Q(t):=e^{EL^{-1}t}$.
We suppose that $M_{0}:=\sup_{t\geq0}\Vert Q(t)\Vert<\infty$ and, for short,
we denote $C_{1}=\Vert L^{-1}\Vert$ and $C_{2}=\Vert M^{-1}\Vert$.

According to previous definitions, it is suitable to rewrite problem
\eqref{eq:1.1}--\eqref{eq:1.2} as the equivalent integral equation
\begin{equation}
\label{eq:2.1}
Lu(t)=Lu(0)+\frac{1}{\Gamma(\alpha)}\int_{0}^{t}(t-s)^{\alpha-1}[Eu(s)+f(s, W(s))]ds,
\end{equation}
provided the integral in \eqref{eq:2.1} exists for a.a. $t\in J$.

\begin{remark}
\label{Remark 2.2}
Note that:
\begin{itemize}
\item[(i)] For the nonlocal condition, the function $u(0)$ is dependent on $t$.
\item[(ii)] $^{L}D^{1-\alpha}_{t}[Mu(0)]$ is well defined, i.e., if $\alpha=1$
and $M$ is the identity, then \eqref{eq:1.2} reduces to the usual nonlocal condition.
\item[(iii)] Function $u(0)$ takes the form
$$
M^{-1}v_{0}+\frac{1}{\Gamma(1-\alpha)}\int_{0}^{t}\frac{M^{-1}[u_{0}+h(u(s))]}{(t-s)^{\alpha}}ds,
$$
where $Mu(0)|_{t=0}=v_{0}$.
\item[(iv)] The explicit and implicit integrals given in \eqref{eq:2.1} exist (taken in Bochner's sense).
\end{itemize}
\end{remark}

Throughout the paper, $A=EL^{-1}: D(A)\subset Y\rightarrow Y$
will be the infinitesimal generator of a compact analytic
semigroup of uniformly bounded linear operators $Q(\cdot)$. Then, there exists a constant
$M_{0}\geq1$ such that $\Vert Q(t)\Vert\leq M_{0}$ for $t\geq0$. Without loss of generality,
we assume that $0\in \rho(A)$, the resolvent set of $A$. Then it is possible to define
the fractional power $A^{q}$, $0 < q \leq 1$, as a closed linear operator on its domain $D(A^{q})$
with inverse $A^{-q}$. Furthermore, the subspace $D(A^{q})$ is dense in $X$ and the expression
$\Vert u\Vert_{q}=\Vert A^{q}u\Vert, u\in D(A^{q})$ defines a norm on $D(A^{q})$.
Hereafter, we denote by $X_{q}$ the Banach space $D(A^{q})$ normed with $\Vert u\Vert_{q}$.

\begin{lemma}[See \cite{AMA.33}]
\label{Lemma 2.1}
Let $A$ be the infinitesimal generator of an analytic semigroup $Q(t)$.
If $0\in \rho(A)$, then
\begin{itemize}
\item[(a)] $Q(t): X\rightarrow D(A^{q})$ for every $t>0$ and $q\geq0$.
\item[(b)] For every $u\in D(A^{q})$, we have $Q(t)A^{q}u=A^{q}Q(t)u$.
\item[(c)] For every $t>0$, the operator $A^{q}Q(t)$ is bounded and
$\Vert A^{q}Q(t)\Vert\leq M_{q}t^{-q}e^{-\omega t}$.
\item[(d)] If $0<q\leq1$ and $u\in D(A^{q})$, then
$\Vert Q(t)u-u\Vert\leq C_{q}t^{q}\Vert A^{q}u\Vert$.
\end{itemize}
\end{lemma}

\begin{remark}
\label{Remark 2.3}
Note that:
\begin{itemize}
\item[(i)] $D(A^{q})$ is a Banach space with the norm
$\Vert u\Vert_{q}=\Vert A^{q}u\Vert$ for $u\in D(A^{q})$.
\item[(ii)] If $0<p\leq q\leq1$, then $D(A^{q})\hookrightarrow D(A^{p})$.
\item[(iii)] $A^{-q}$ is a bounded linear operator in $X$ with $D(A^{q})=Im(A^{-q})$.
\end{itemize}
\end{remark}

\begin{remark}
\label{Remark 2.4}
Observe, as in \cite{AMA.35}, that by Lemma~\ref{Lemma 2.1} (a) and (b),
the restriction $Q_{q}(t)$ of $Q(t)$ to $X_{q}$ is exactly the part of $Q(t)$
in $X_{q}$. Let $u\in X_{q}$. Since $\Vert Q(t)u\Vert_{q}\leq\Vert A^{q}Q(t)u\Vert
=\Vert Q(t)A^{q}u\Vert\leq\Vert Q(t)\Vert\Vert A^{q}u\Vert=\Vert Q(t)\Vert\Vert u\Vert_{q}$,
and as $t$ decreases to $0^{+}$, $\Vert Q(t)u-u\Vert_{q}
=\Vert A^{q}Q(t)u-A^{q}u\Vert=\Vert Q(t)A^{q}u-A^{q}u\Vert\rightarrow0$
for all $u\in X_{q}$, it follows that $\lbrace Q(t), t\geq0\rbrace$ is a family of strongly
continuous semigroups on $X_{q}$ and $\Vert Q_{q}(t)\Vert\leq \Vert Q(t)\Vert\leq M_{0}$ for all $t\geq 0$.
\end{remark}

In the sequel, we will also use $\Vert \phi\Vert_{L^{p}(J, \mathbb{R}^{+})}$
to denote the $L^{p}(J, \mathbb{R}^{+})$ norm of $\phi$ whenever
$\phi\in L^{p}(J, \mathbb{R}^{+})$ for some $p$ with $1<p<\infty$.
We will set $q\in (0, 1)$ and denote by $\Omega_{q}$ the Banach space $C(J, X_{q})$
endowed with supnorm given by $\Vert u\Vert_{\infty}=\sup_{t\in J}\Vert u\Vert_{q}$
for $u\in \Omega_{q}$.

Motivated by \cite{AMA.21,AMA.30,AMA.36}, we give the definition
of mild solution to \eqref{eq:1.1}--\eqref{eq:1.2}.

\begin{definition}
\label{Definition 2.4}
A function $u\in \Omega_{q}$ is called a mild solution of system
\eqref{eq:1.1}--\eqref{eq:1.2} if it satisfies the following integral equation:
\begin{equation*}
u(t)=S_{\alpha}(t)LM^{-1}\left[v_{0}+\frac{1}{\Gamma(1-\alpha)}
\int_{0}^{t}\frac{[u_{0}+h(u(s))]}{(t-s)^{\alpha}}ds\right]
+\int_{0}^{t}(t-s)^{\alpha-1}T_{\alpha}(t-s)f(s, W(s))ds,
\end{equation*}
where
$$
S_{\alpha}(t)=\int_{0}^{\infty}L^{-1}\zeta_{\alpha}(\theta)Q(t^{\alpha}\theta)d\theta,
\quad T_{\alpha}(t)=\alpha\int_{0}^{\infty}L^{-1}\theta\zeta_{\alpha}(\theta)Q(t^{\alpha}\theta)d\theta,
$$
$$
\zeta_{\alpha}(\theta)=\frac{1}{\alpha}\theta^{-1-\frac{1}{\alpha}}\varpi_{\alpha}(\theta^{-\frac{1}{\alpha}})\geq0,
\quad \varpi_{\alpha}(\theta)=\frac{1}{\pi}\sum_{n=1}^{\infty}(-1)^{n-1}\theta^{-\alpha n-1}
\frac{\Gamma(n\alpha+1)}{n!}\sin (n\pi\alpha), \theta\in (0, \infty),
$$
with $\zeta_{\alpha}$ the probability density function defined on $(0, \infty)$, that is,
$\zeta_{\alpha}(\theta)\geq0, \theta\in(0, \infty)$
and $\int_{0}^{\infty}\zeta_{\alpha}(\theta)d\theta=1$.
\end{definition}

\begin{remark}
\label{Remark 2.5}
For $v\in [0, 1]$, ones has
$$
\int_{0}^{\infty}\theta^{v}\zeta_{\alpha}(\theta)d\theta
=\int_{0}^{\infty}\theta^{-\alpha v}\varpi_{\alpha}(\theta)d\theta
=\frac{\Gamma(1+v)}{\Gamma(1+\alpha v)}
$$
(see \cite{AMA.37}).
\end{remark}

\begin{lemma}[See \cite{AMA.30,AMA.36,AMA.37}]
\label{Lemma 2.2}
The operators $S_{\alpha}(t)$ and $T_{\alpha}(t)$ have the following properties:
\begin{itemize}
\item[(a)] For any fixed $t\geq0$, the operators $S_{\alpha}(t)$ and $T_{\alpha}(t)$
are linear and bounded, i.e., for any $u\in X$, $\Vert S_{\alpha}(t)u\Vert\leq C_{1}M_{0}\Vert u\Vert$
and $\Vert T_{\alpha}(t)u\Vert\leq \frac{C_{1}M_{0}}{\Gamma (\alpha)}\Vert u\Vert$.
\item[(b)] $\lbrace S_{\alpha}(t), t\geq0\rbrace$ and $\lbrace T_{\alpha}(t)$,
$t\geq0\rbrace$ are strongly continuous, i.e., for $u\in X$ and $0\leq t_{1}<t_{2}\leq a$, we have
$\Vert S_{\alpha}(t_{2})u-S_{\alpha}(t_{1})u\Vert\rightarrow 0$
and $\Vert T_{\alpha}(t_{2})u-T_{\alpha}(t_{1})u\Vert\rightarrow 0$ as $t_{1}\rightarrow t_{2}$.
\item[(c)] For every $t>0$, $S_{\alpha}(t)$ and $T_{\alpha}(t)$ are compact operators.
\item[(d)] For any $u\in X$, $p\in (0, 1)$ and $q\in (0, 1)$, we have
$AT_{\alpha}(t)u=A^{1-p}T_{\alpha}(t)A^{p}u$, $t\in J$, and $\Vert A^{q}T_{\alpha}(t)\Vert
\leq \frac{\alpha C_{1}M_{q}\Gamma(2-q)}{\Gamma(1+\alpha(1-q))}t^{-q\alpha}$, $0<t\leq a$.
\item[(e)] For fixed $t\geq0$ and any $u\in X_{q}$, we have
$\Vert S_{\alpha}(t)u\Vert_{q}\leq C_{1}M_{0}\Vert u\Vert_{q}$
and $\Vert T_{\alpha}(t)u\Vert_{q}\leq \frac{C_{1}M_{0}}{\Gamma (\alpha)}\Vert u\Vert_{q}$.
\item[(f)] $S_{\alpha}(t)$ and $T_{\alpha}(t)$, $t>0$, are uniformly continuous, that is,
for each fixed $t>0$ and $\epsilon>0$ there exists $g>0$ such that
$\Vert S_{\alpha}(t+\epsilon)-S_{\alpha}(t)\Vert_{q}<\epsilon$
for $t+\epsilon\geq0$ and $\vert \epsilon\vert<g$,
$\Vert T_{\alpha}(t+\epsilon)-T_{\alpha}(t)\Vert_{q}<\epsilon$
for $t+\epsilon\geq0$ and $\vert \epsilon\vert<g$.
\end{itemize}
\end{lemma}

\begin{lemma}[See \cite{AMA.38}]
\label{Lemma 2.3}
For each $\psi\in L^{p}(J, X)$ with $1\leq p<\infty$,
$$
\lim\limits_{g\rightarrow0}\int_{0}^{a}\Vert \psi(t+g)-\psi(t)\Vert^{p}dt=0,
$$
where $\psi(s)=0$ for $s\notin J$.
\end{lemma}

\begin{lemma}[See \cite{AMA.37}]
\label{Lemma 2.4}
A measurable function $G: J\rightarrow X$ is a Bochner integral
if $\Vert G\Vert$ is Lebesgue integrable.
\end{lemma}


\section{Main results}
\label{sec:3}

Our first result provides existence of mild solutions
to system \eqref{eq:1.1}--\eqref{eq:1.2}. To prove that,
we make use of the following assumptions:
\begin{itemize}
\item[($F_1$)] The linear closed operators $\lbrace B_{i}(t)\rbrace_{i=\overline{1,r}}$
are defined on dense sets $S_{1},\ldots,S_{r}\supset D(A)$, respectively from $X_{q}$ into $Y$.
\item[($F_2$)] The function $f: J\times X^{r}_{q}\rightarrow Y$ satisfies:
for each $W\in X^{r}_{q}$, in particular, for every element $u\in \cap_{i}S_{i}$, $i=1,\ldots,r$,
the function $t\rightarrow f(t, W(t))$ is measurable.
\item[($F_3$)] For arbitrary $u, u^{*}\in X_{q}$ satisfying $\Vert u\Vert_{q}, \Vert u^{*}\Vert_{q}\leq \rho$,
there exists a constant $L_{f}(\rho)>0$ and functions $m_{i}\in L^{1}(J, \mathbb{R}^{+})$ such that
$$
\Vert f(t, W)-f(t, W^{*})\Vert\leq L_{f}(\rho)[m_{1}(t)+\cdots+m_{r}(t)]\Vert u-u^{*}\Vert_{q}
$$
for almost all $t\in J$. Here, $W^{*}(t)=(B_{1}(t)u^{*}(t),\ldots,B_{r}(t)u^{*}(t))$, $i=1,\ldots,r$.
\item[($F_4$)] There exists a constant $a_{f}>0$ such that
$$
\Vert f(t, W)\Vert\leq a_{f}(1+r\Vert u\Vert_{q})
\ \text{ for all }\  W\in X^{r}_{q} \ \text{ and }\  t\in J.
$$
\item[($F_{5}$)] The function $h: C(J: X_{q})\rightarrow X_{q}$
is Lipschitz continuous and bounded in $X_{q}$, i.e.,
for all $u, v\in C(J, X_{q})$ there exist constants $k_{1}, k_{2}>0$ such that
$$
\Vert h(u)-h(v)\Vert_{q}\leq k_{1}\Vert u-v\Vert_{q}
\  \text{ and } \ \Vert h(u)\Vert_{q}\leq k_{2}.
$$
\end{itemize}

\begin{theorem}
\label{Theorem 3.1}
Assume hypotheses ($F_1$)--($F_5$) are satisfied.
If $u_{0}\in X_{q}$ and $\alpha q<1$ for some $\frac{1}{2}<\alpha<1$,
then system \eqref{eq:1.1}--\eqref{eq:1.2} has a mild solution on $J$.
\end{theorem}

The following lemmas are used in the proof of Theorem~\ref{Theorem 3.1}.

\begin{lemma}
\label{Lemma 3.1}
Let operator $P: \Omega_{q}\rightarrow\Omega_{q}$ be given by
\begin{multline}
\label{eq:op:P}
(Pu)(t)=S_{\alpha}(t)LM^{-1}\left[v_{0}+\frac{1}{\Gamma(1-\alpha)}
\int_{0}^{t}\frac{[u_{0}+h(u(s))]}{(t-s)^{\alpha}}ds\right]\\
+\int_{0}^{t}(t-s)^{\alpha-1}T_{\alpha}(t-s)f(s, W(s))ds.
\end{multline}
Then, the operator $P$ satisfies $Pu\in \Omega_{q}$.
\end{lemma}

\begin{proof}
Let $0\leq t_{1}<t_{2}\leq a$ and $\alpha q<\frac{1}{2}$. We have
\begin{equation*}
\begin{split}
&\Vert (Pu)(t_{1})-(Pu)(t_{2})\Vert_{q}\\
&=\left\Vert[S_{\alpha}(t_{1})-S_{\alpha}(t_{2})]LM^{-1}\left[v_{0}+\frac{1}{\Gamma(1-\alpha)}
\int_{0}^{t_{1}}(t_{1}-s)^{-\alpha}[u_{0}+h(u(s))]ds\right]\right\Vert_{q}\\
&\quad +\left\Vert S_{\alpha}(t_{2})LM^{-1}\left[\frac{1}{\Gamma(1-\alpha)}
\int_{0}^{t_{1}}[(t_{1}-s)^{-\alpha}-(t_{2}-s)^{-\alpha}][u_{0}+h(u(s))]ds\right]\right\Vert_{q}\\
&\quad +\left\Vert S_{\alpha}(t_{2})LM^{-1}\left[\frac{1}{\Gamma(1-\alpha)}
\int_{t_{1}}^{t_{2}}(t_{2}-s)^{-\alpha}[u_{0}+h(u(s))]ds\right]\right\Vert_{q}\\
&\quad +\int_{0}^{t_{1}}(t_{1}-s)^{\alpha-1}\Vert T_{\alpha}(t_{1}-s)f(s, W(s))
-T_{\alpha}(t_{2}-s)f(s, W(s))\Vert_{q} ds\\
&\quad +\int_{0}^{t_{1}}\vert(t_{1}-s)^{\alpha-1}-(t_{2}-s)^{\alpha-1}
\vert\Vert T_{\alpha}(t_{2}-s)f(s, W(s))\Vert_{q} ds\\
&\quad +\int_{t_{1}}^{t_{2}}(t_{2}-s)^{\alpha-1}\Vert T_{\alpha}(t_{2}-s)f(s, W(s))\Vert_{q} ds.
\end{split}
\end{equation*}
We use Lemma~\ref{Lemma 2.2}, and fractional power of operators, to get
\begin{equation*}
\begin{split}
\Vert (Pu)(t_{1})-(Pu)(t_{2})\Vert_{q}
\leq & C_{2}\Vert L\Vert \left[\Vert v_{0}\Vert_{q}+(k_{2}+\Vert u_{0}
\Vert_{q})\frac{t_{1}^{1-\alpha}}{\Gamma(2-\alpha)}\right]
\Vert S_{\alpha}(t_{1})-S_{\alpha}(t_{2})\Vert_{q}\\
&+C_{1}C_{2}M_{0}\Vert L\Vert \left[(k_{2}+\Vert u_{0}\Vert_{q})
\frac{(t_{2}-t_{1})^{1-\alpha}+t_{1}^{1-\alpha}-t_{2}^{1-\alpha}}{\Gamma(2-\alpha)}\right]\\
&+C_{1}C_{2}M_{0}\Vert L\Vert \left[(k_{2}+\Vert u_{0}\Vert_{q})
\frac{(t_{2}-t_{1})^{1-\alpha}}{\Gamma(2-\alpha)}\right]\\
&+\int_{0}^{t_{1}}(t_{1}-s)^{\alpha-1}\Vert A^{q}[T_{\alpha}(t_{1}-s)
-T_{\alpha}(t_{2}-s)]\Vert \Vert f(s, W(s))\Vert ds\\
&+\int_{0}^{t_{1}}\vert(t_{1}-s)^{\alpha-1}-(t_{2}-s)^{\alpha-1}
\vert\Vert A^{q}T_{\alpha}(t_{2}-s)\Vert \Vert f(s, W(s))\Vert ds\\
&+\int_{t_{1}}^{t_{2}}(t_{2}-s)^{\alpha-1}\Vert A^{q}T_{\alpha}(t_{2}-s)
\Vert \Vert f(s, W(s))\Vert ds
\end{split}
\end{equation*}
\begin{equation*}
\begin{split}
\leq & C_{2}\Vert L\Vert \left[\Vert v_{0}\Vert_{q}+(k_{2}+\Vert u_{0}\Vert_{q})
\frac{t_{1}^{1-\alpha}}{\Gamma(2-\alpha)}\right]\Vert S_{\alpha}(t_{1})-S_{\alpha}(t_{2})\Vert_{q}\\
&+C_{1}C_{2}M_{0}\Vert L\Vert \left[(k_{2}+\Vert u_{0}\Vert_{q})
\frac{(t_{2}-t_{1})^{1-\alpha}+t_{1}^{1-\alpha}-t_{2}^{1-\alpha}}{\Gamma(2-\alpha)}\right]\\
&+C_{1}C_{2}M_{0}\Vert L\Vert \left[(k_{2}+\Vert u_{0}\Vert_{q})
\frac{(t_{2}-t_{1})^{1-\alpha}}{\Gamma(2-\alpha)}\right]\\
&+\frac{\alpha C_{1}M_{q}\Gamma(2-q)}{\Gamma(1+\alpha(1-q))}\Vert f\Vert_{C(J, X)}
\int_{0}^{t_{1}}(t_{1}-s)^{\alpha-1}\vert (t_{1}-s)^{-q\alpha}-(t_{2}-s)^{-q\alpha}\vert ds\\
&+\frac{\alpha C_{1}M_{q}\Gamma(2-q)}{\Gamma(1+\alpha(1-q))}
\int_{0}^{t_{1}}\vert(t_{1}-s)^{\alpha-1}-(t_{2}-s)^{\alpha-1}
\vert(t_{2}-s)^{-q\alpha}\Vert f(s, W(s))\Vert ds\\
&+\frac{\alpha C_{1}M_{q}\Gamma(2-q)}{\Gamma(1+\alpha(1-q))}
\int_{t_{1}}^{t_{2}}(t_{2}-s)^{-q\alpha+\alpha-1}\Vert f(s, W(s))\Vert ds.
\end{split}
\end{equation*}
From Lemma~\ref{Lemma 2.2} and H\"older's inequality, one can
deduce the following inequality:
\begin{equation*}
\begin{split}
\Vert (Pu)&(t_{1})-(Pu)(t_{2})\Vert_{q}\\
&\leq C_{2}\Vert L\Vert \left[\Vert v_{0}\Vert_{q}+(k_{2}+\Vert u_{0}\Vert_{q})
\frac{t_{1}^{1-\alpha}}{\Gamma(2-\alpha)}\right]\Vert S_{\alpha}(t_{1})-S_{\alpha}(t_{2})\Vert_{q}\\
&\quad + C_{1}C_{2}M_{0}\Vert L\Vert \left[(k_{2}+\Vert u_{0}\Vert_{q})
\frac{(t_{2}-t_{1})^{1-\alpha}+t_{1}^{1-\alpha}-t_{2}^{1-\alpha}}{\Gamma(2-\alpha)}\right]\\
&\quad + C_{1}C_{2}M_{0}\Vert L\Vert \left[(k_{2}+\Vert u_{0}\Vert_{q})
\frac{(t_{2}-t_{1})^{1-\alpha}}{\Gamma(2-\alpha)}\right]\\
&\quad + \frac{\alpha C_{1}M_{q}\Gamma(2-q)}{\Gamma(1+\alpha(1-q))}
\Vert f\Vert_{C(J, X)} \Biggl[\left(\int_{0}^{t_{1}}\vert (t_{1}-s)^{-q\alpha}
-(t_{2}-s)^{-q\alpha}\vert^{2} ds\right)^{\frac{1}{2}}\\
&\quad \times\left(\int_{0}^{t_{1}}(t_{1}-s)^{2(\alpha-1)} ds\right)^{\frac{1}{2}}
+\left(\int_{0}^{t_{1}}\vert(t_{1}-s)^{\alpha-1}
-(t_{2}-s)^{\alpha-1}\vert^{2} ds\right)^{\frac{1}{2}}\\
&\quad \times\left(\int_{0}^{t_{1}}(t_{2}-s)^{-2q\alpha} ds\right)^{\frac{1}{2}}
+\frac{1}{\alpha(1-q)}(t_{2}-t_{1})^{\alpha(1-q)}\Biggl]\\
&\leq C_{2}\Vert L\Vert \left[\Vert v_{0}\Vert_{q}+(k_{2}+\Vert u_{0}\Vert_{q})
\frac{t_{1}^{1-\alpha}}{\Gamma(2-\alpha)}\right]\Vert S_{\alpha}(t_{1})-S_{\alpha}(t_{2})\Vert_{q}\\
&\quad + C_{1}C_{2}M_{0}\Vert L\Vert \left[(k_{2}+\Vert u_{0}\Vert_{q})
\frac{(t_{2}-t_{1})^{1-\alpha}+t_{1}^{1-\alpha}-t_{2}^{1-\alpha}}{\Gamma(2-\alpha)}\right]\\
&\quad + C_{1}C_{2}M_{0}\Vert L\Vert \left[(k_{2}+\Vert u_{0}\Vert_{q})
\frac{(t_{2}-t_{1})^{1-\alpha}}{\Gamma(2-\alpha)}\right]\\
&\quad + \frac{\alpha C_{1}M_{q}\Gamma(2-q)}{\Gamma(1+\alpha(1-q))}\Vert f\Vert_{C(J, X)}
\Biggl[\sqrt{\frac{1}{2\alpha-1}}t_{1}^{\alpha-\frac{1}{2}}\left(
\int_{0}^{a}\vert (t_{1}-s)^{-q\alpha}-(t_{2}-s)^{-q\alpha}\vert^{2} ds\right)^{\frac{1}{2}}\\
&\quad +\left(\int_{0}^{a}\vert(t_{1}-s)^{\alpha-1}-(t_{2}-s)^{\alpha-1}\vert^{2} ds\right)^{\frac{1}{2}}
\sqrt{\frac{1}{1-2q\alpha}}\biggl(t_{2}^{1-2q\alpha}-(t_{2}-t_{1})^{1-2q\alpha}\biggr)^{\frac{1}{2}}\\
&\quad +\frac{1}{\alpha(1-q)}(t_{2}-t_{1})^{\alpha(1-q)}\Biggr],
\end{split}
\end{equation*}
which means that $Pu\in \Omega_{q}$.
\end{proof}

\begin{lemma}
\label{Lemma 3.2.}
The operator $P$ given by \eqref{eq:op:P} is continuous on $\Omega_{q}$.
\end{lemma}

\begin{proof}
Let $u, u^{*}\in \Omega_{q}$ and $\Vert u-u^{*}\Vert_{\infty}\leq1$.
Then, $\Vert u\Vert_{\infty}\leq1+\Vert u^{*}\Vert_{\infty}=\rho$ and
\begin{equation*}
\begin{split}
\Vert(Pu)(t)-(Pu^{*})(t)\Vert_{q}
&=\left\Vert S_{\alpha}(t)LM^{-1}\left[
\frac{1}{\Gamma(1-\alpha)}\int_{0}^{t}(t-s)^{-\alpha}[h(u)-h(u^{*})]ds\right]\right\Vert_{q}\\
&\quad +\int_{0}^{t}(t-s)^{\alpha-1}\Vert T_{\alpha}(t-s)[f(s, W(s))-f(s, W^{*}(s))]\Vert_{q} ds\\
&\leq\Vert S_{\alpha}(t)LM^{-1}\Vert\frac{1}{\Gamma(1-\alpha)}
\int_{0}^{t}(t-s)^{-\alpha}\left\Vert A^{q}[h(u)-h(u^{*})]\right\Vert ds\\
&\quad +\int_{0}^{t}(t-s)^{\alpha-1}\Vert A^{q}T_{\alpha}(t-s)\Vert \Vert f(s, W(s))-f(s, W^{*}(s))\Vert ds\\
&\leq C_{1}C_{2}k_{1}M_{0}\Vert L\Vert\frac{a^{1-\alpha}}{\Gamma(2-\alpha)} \Vert u-u^{*}\Vert_{q}\\
&\quad + L_{f}(\rho)\sum\limits_{i=1}^{r}m_{i}(t)\frac{\alpha C_{1}M_{q}\Gamma(2-q)}{\Gamma(1+\alpha(1-q))}
\int_{0}^{t}(t-s)^{-q\alpha+\alpha-1}\Vert u-u^{*}\Vert_{q}ds\\
&\leq C_{1}C_{2}k_{1}M_{0}\Vert L\Vert\frac{a^{1-\alpha}}{\Gamma(2-\alpha)} \Vert u-u^{*}\Vert_{\infty}\\
&\quad +L_{f}(\rho)\sum\limits_{i=1}^{r}m_{i}(t)\frac{\alpha C_{1}M_{q}\Gamma(2-q)}{\Gamma(1+\alpha(1-q))}
\frac{1}{\alpha(1-q)}t^{\alpha(1-q)}\Vert u-u^{*}\Vert_{\infty}.
\end{split}
\end{equation*}
Therefore,
\begin{multline*}
\Vert(Pu)(t)-(Pu^{*})(t)\Vert_{\infty}
\leq C_{1}C_{2}k_{1}M_{0}
\Vert L\Vert\frac{a^{1-\alpha}}{\Gamma(2-\alpha)} \Vert u-u^{*}\Vert_{\infty}\\
+L_{f}(\rho)\sum\limits_{i=1}^{r}m_{i}(t)
\frac{\alpha C_{1}M_{q}\Gamma(2-q)}{\Gamma(1+\alpha(1-q))}
\frac{1}{\alpha(1-q)}t^{\alpha(1-q)}\Vert u-u^{*}\Vert_{\infty}
\end{multline*}
and we conclude that $P$ is continuous.
\end{proof}

\begin{lemma}
\label{Lemma 3.3}
The operator $P$ given by \eqref{eq:op:P} is compact.
\end{lemma}

\begin{proof}
Let $\Sigma$ be a bounded subset of $\Omega_{q}$. Then there exists a constant $\eta$
such that $\Vert u\Vert_{\infty}\leq \eta$ for all $u\in \Sigma$.
By ($F_4$), there exists a constant $\tau$ such that
$\Vert f(t, W(t))\Vert\leq a_{f}(1+r\eta)=\tau$. Then $P\Sigma$
is a bounded subset of $\Omega_{q}$. In fact, let $u\in \Sigma$.
Using Lemma~\ref{Lemma 2.2} (a) and (d), we get
\begin{equation*}
\begin{split}
\Vert(Pu)(t)\Vert_{q}
&\leq \left\Vert S_{\alpha}(t)LM^{-1}\left[v_{0}
+\frac{1}{\Gamma(1-\alpha)}\int_{0}^{t}
\frac{[u_{0}+h(u(s))]}{(t-s)^{\alpha}}ds\right]\right\Vert_{q}\\
&\quad +\int_{0}^{t}(t-s)^{\alpha-1}\Vert T_{\alpha}(t-s)f(s, W(s))\Vert_{q}ds\\
&\leq C_{1}C_{2}M_{0}\Vert L\Vert\left[\Vert v_{0}\Vert_{q}
+\frac{a^{1-\alpha}}{\Gamma(2-\alpha)}(k_{2}+\Vert u_{0}\Vert_{q})\right]\\
&\quad +\int_{0}^{t}(t-s)^{\alpha-1}\Vert A^{q}T_{\alpha}(t-s)
\Vert \Vert f(s, W(s))\Vert ds\\
&\leq C_{1}C_{2}M_{0}\Vert L\Vert\left[\Vert v_{0}\Vert_{q}
+\frac{a^{1-\alpha}}{\Gamma(2-\alpha)}(k_{2}+\Vert u_{0}\Vert_{q})\right]\\
&\quad +\frac{\alpha C_{1}M_{q}\Gamma(2-q)}{\Gamma(1+\alpha(1-q))}\tau
\int_{0}^{t}(t-s)^{-q\alpha+\alpha-1}ds\\
&\leq C_{1}C_{2}M_{0}\Vert L\Vert\left[\Vert v_{0}\Vert_{q}
+\frac{a^{1-\alpha}}{\Gamma(2-\alpha)}(k_{2}+\Vert u_{0}\Vert_{q})\right]\\
&\quad +\frac{\alpha C_{1}M_{q}\Gamma(2-q)}{\Gamma(1+\alpha(1-q))}\tau
\frac{1}{\alpha(1-q)}t^{\alpha(1-q)}.
\end{split}
\end{equation*}
Then, we obtain
$$
\Vert(Pu)(t)\Vert_{\infty}\leq C_{1}C_{2}M_{0}\Vert L\Vert\left[\eta
+\frac{a^{1-\alpha}}{\Gamma(2-\alpha)}(k_{2}+\eta)\right]+\frac{\alpha C_{1}
M_{q}\Gamma(2-q)}{\Gamma(1+\alpha(1-q))}\frac{\tau a^{\alpha(1-q)}}{\alpha(1-q)}.
$$
We conclude that $P\Sigma$ is bounded.
Define $\Pi=P\Sigma$ and $\Pi(t)=\lbrace (Pu)(t)\vert u\in \Sigma\rbrace$ for $t\in J$.
Obviously, $\Pi(0)=\lbrace (Pu)(0)\vert u\in \Sigma\rbrace$ is compact.
For each $g\in (0, t)$, $t\in (0, a]$, and arbitrary $\delta>0$, let us define
$\Pi_{g,\delta}(t)=\lbrace (P_{g,\delta}u)(t)\vert u\in \Sigma\rbrace$,
where
\begin{equation*}
\begin{split}
(P_{g,\delta}u)(t)
&=Q(g^{\alpha}\delta)\int_{\delta}^{\infty}L^{-1}\zeta_{\alpha}(\theta)Q(t^{\alpha}\theta
-g^{\alpha}\delta)LM^{-1}\left[v_{0}+\frac{1}{\Gamma(1-\alpha)}\int_{0}^{t-g}
\frac{[u_{0}+h(u(s))]}{(t-s)^{\alpha}}ds\right]d\theta\\
&\quad + Q(g^{\alpha}\delta)\int_{0}^{t-g}(t-s)^{\alpha-1}\left(\alpha
\int_{\delta}^{\infty}L^{-1}\theta\zeta_{\alpha}(\theta)Q((t-s)^{\alpha}\theta
-g^{\alpha}\delta)d\theta\right)f(s, W(s))ds\\
&=\int_{\delta}^{\infty}L^{-1}\zeta_{\alpha}(\theta)Q(t^{\alpha}\theta)LM^{-1}\left[
v_{0}+\frac{1}{\Gamma(1-\alpha)}\int_{0}^{t-g}\frac{[u_{0}+h(u(s))]}{(t-s)^{\alpha}}ds\right]d\theta\\
&\quad +\alpha\int_{0}^{t-g}\int_{\delta}^{\infty}\theta(t-s)^{\alpha-1}
L^{-1}\zeta_{\alpha}(\theta)Q((t-s)^{\alpha}\theta)f(s, W(s))d\theta ds.
\end{split}
\end{equation*}
Then, since the operator $Q(g^{\alpha}\delta)$, $g^{\alpha}\delta>0$, is compact
in $X_{q}$, the sets $\lbrace (P_{g,\delta}u)(t)\vert u\in \Sigma\rbrace$ are relatively compact
in $X_{q}$. This comes from the following inequalities:
\begin{equation*}
\begin{split}
\Vert(Pu)&(t)-(P_{g,\delta}u)(t)\Vert_{q}\\
&\leq\biggl\Vert\int_{0}^{\delta}L^{-1}\zeta_{\alpha}(\theta)Q(t^{\alpha}\theta)LM^{-1}\left[
v_{0}+\frac{1}{\Gamma(1-\alpha)}\int_{0}^{t}\frac{[u_{0}+h(u(s))]}{(t-s)^{\alpha}}ds\right]d\theta\biggr\Vert_{q}\\
&\quad +\biggl\Vert\int_{\delta}^{\infty}L^{-1}\zeta_{\alpha}(\theta)Q(t^{\alpha}\theta)LM^{-1}\left[
v_{0}+\frac{1}{\Gamma(1-\alpha)}\int_{t-g}^{t}\frac{[u_{0}+h(u(s))]}{(t-s)^{\alpha}}ds\right]d\theta\biggr\Vert_{q}\\
&\quad +\biggl\Vert\int_{\delta}^{\infty}L^{-1}\zeta_{\alpha}(\theta)Q(t^{\alpha}\theta)LM^{-1}\left[
v_{0}+\frac{1}{\Gamma(1-\alpha)}\int_{0}^{t-g}\frac{[u_{0}+h(u(s))]}{(t-s)^{\alpha}}ds\right]d\theta\\
&\quad -\int_{\delta}^{\infty}L^{-1}\zeta_{\alpha}(\theta)Q(t^{\alpha}\theta)LM^{-1}\left[v_{0}
+\frac{1}{\Gamma(1-\alpha)}\int_{0}^{t-g}\frac{[u_{0}+h(u(s))]}{(t-s)^{\alpha}}ds\right]d\theta\biggr\Vert_{q}\\
&\quad +\alpha\left\Vert\int_{0}^{t}\int_{0}^{\delta}\theta(t-s)^{\alpha-1}
L^{-1}\zeta_{\alpha}(\theta)Q((t-s)^{\alpha}\theta)f(s, W(s))d\theta ds\right\Vert_{q}\\
&\quad +\alpha\biggl\Vert\int_{0}^{t}\int_{\delta}^{\infty}\theta(t-s)^{\alpha-1}
L^{-1}\zeta_{\alpha}(\theta)Q((t-s)^{\alpha}\theta)f(s, W(s))d\theta ds\\
&\quad -\int_{0}^{t-g}\int_{\delta}^{\infty}\theta(t-s)^{\alpha-1}L^{-1}\zeta_{\alpha}(\theta)
Q((t-s)^{\alpha}\theta)f(s, W(s))d\theta ds\biggr\Vert_{q}\\
&\leq\int_{0}^{\delta}\Vert L^{-1}\zeta_{\alpha}(\theta)Q(t^{\alpha}\theta)LM^{-1}
\Vert \biggl\Vert A^{q} \left[v_{0}+\frac{1}{\Gamma(1-\alpha)}
\int_{0}^{t}\frac{[u_{0}+h(u(s))]}{(t-s)^{\alpha}}ds\right]\biggr\Vert d\theta\\
&\quad +\int_{\delta}^{\infty}\Vert L^{-1}\zeta_{\alpha}(\theta)Q(t^{\alpha}\theta)LM^{-1}
\Vert \biggl\Vert A^{q}\left[v_{0}+\frac{1}{\Gamma(1-\alpha)}\int_{t-g}^{t}
\frac{[u_{0}+h(u(s))]}{(t-s)^{\alpha}}ds\right]\biggr\Vert d\theta\\
&\quad +\alpha\int_{0}^{t}\int_{0}^{\delta}\theta(t-s)^{\alpha-1}\Vert L^{-1}
\Vert\zeta_{\alpha}(\theta)\Vert A^{q}Q((t-s)^{\alpha}\theta)\Vert \Vert f(s, W(s))\Vert d\theta ds\\
&\quad +\alpha\int_{t-g}^{t}\int_{\delta}^{\infty}\theta(t-s)^{\alpha-1}
\Vert L^{-1}\Vert\zeta_{\alpha}(\theta)\Vert A^{q}Q((t-s)^{\alpha}\theta)\Vert \Vert f(s, W(s))\Vert d\theta ds
\end{split}
\end{equation*}
\begin{equation*}
\begin{split}
&\leq C_{1}C_{2}M_{0}\Vert L\Vert\left[\Vert v_{0}\Vert_{q}+(k_{2}
+\Vert u_{0}\Vert_{q})\frac{t^{1-\alpha}}{\Gamma(2-\alpha)}\right]\int_{0}^{\delta}\zeta_{\alpha}(\theta)d\theta\\
&\quad +C_{1}C_{2}M_{0}\Vert L\Vert \left[\Vert v_{0}\Vert_{q}+(k_{2}
+\Vert u_{0}\Vert_{q})\frac{g^{1-\alpha}}{\Gamma(2-\alpha)}\right]
\int_{\delta}^{\infty}\zeta_{\alpha}(\theta)d\theta\\
&\quad +C_{1}M_{q}\alpha\tau\int_{0}^{t}\int_{0}^{\delta}\theta(t-s)^{\alpha-1}
\zeta_{\alpha}(\theta)(t-s)^{-\alpha q}\theta^{-q}d\theta ds\\
&\quad +C_{1}M_{q}\alpha\tau\int_{t-g}^{t}\int_{\delta}^{\infty}\theta(t-s)^{\alpha-1}
\zeta_{\alpha}(\theta)(t-s)^{-\alpha q}\theta^{-q}d\theta ds\\
&\leq C_{1}C_{2}M_{0}\Vert L\Vert\left[\Vert v_{0}\Vert_{q}+(k_{2}
+\Vert u_{0}\Vert_{q})\frac{t^{1-\alpha}}{\Gamma(2-\alpha)}\right]
\int_{0}^{\delta}\zeta_{\alpha}(\theta)d\theta\\
&\quad + C_{1}C_{2}M_{0}\Vert L\Vert\left[\Vert v_{0}\Vert_{q}+(k_{2}
+\Vert u_{0}\Vert_{q})\frac{g^{1-\alpha}}{\Gamma(2-\alpha)}\right]\\
&\quad + C_{1}M_{q}\alpha\tau\int_{0}^{t}\int_{0}^{\delta}\theta^{1-q}(t-s)^{-\alpha q
+\alpha-1}\zeta_{\alpha}(\theta)d\theta ds\\
&\quad + C_{1}M_{q}\alpha\tau\int_{t-g}^{t}\int_{\delta}^{\infty}\theta^{1-q}(t-s)^{-\alpha q
+\alpha-1}\zeta_{\alpha}(\theta)d\theta ds\\
&\leq C_{1}C_{2}M_{0}\Vert L\Vert\left[\Vert v_{0}\Vert_{q}+(k_{2}
+\Vert u_{0}\Vert_{q})\frac{t^{1-\alpha}}{\Gamma(2-\alpha)}\right]
\int_{0}^{\delta}\zeta_{\alpha}(\theta)d\theta\\
&\quad + C_{1}C_{2}M_{0}\Vert L\Vert\left[\Vert v_{0}\Vert_{q}+(k_{2}
+\Vert u_{0}\Vert_{q})\frac{g^{1-\alpha}}{\Gamma(2-\alpha)}\right]\\
&\quad + C_{1}M_{q}\alpha\tau\left(\int_{0}^{t}(t-s)^{-\alpha q+\alpha-1}ds\right)
\int_{0}^{\delta}\theta^{1-q}\zeta_{\alpha}(\theta)d\theta\\
&\quad + C_{1}M_{q}\alpha\tau\frac{\Gamma(2-q)}{\Gamma(1+\alpha(1-q))}\left(
\int_{t-g}^{t}(t-s)^{-\alpha q+\alpha-1}ds\right)
\end{split}
\end{equation*}
and
$$
\int_{0}^{t}(t-s)^{-\alpha q+\alpha-1}ds\leq \frac{1}{\alpha(1-q)}t^{\alpha(1-q)},
~ \int_{t-g}^{t}(t-s)^{-\alpha q+\alpha-1}ds\leq\frac{1}{\alpha(1-q)}g^{\alpha(1-q)},
$$
so that
\begin{equation*}
\begin{split}
\Vert(Pu)(t)-(P_{g,\delta}u)(t)\Vert_{q}
\leq & C_{1}C_{2}M_{0}\Vert L\Vert\left[
\Vert v_{0}\Vert_{q}+(k_{2}+\Vert u_{0}\Vert_{q})\frac{a^{1-\alpha}}{\Gamma(2-\alpha)}\right]
\int_{0}^{\delta}\zeta_{\alpha}(\theta)d\theta\\
&+ C_{1}C_{2}M_{0}\Vert L\Vert\left[\Vert v_{0}\Vert_{q}+(k_{2}
+\Vert u_{0}\Vert_{q})\frac{g^{1-\alpha}}{\Gamma(2-\alpha)}\right]\\
&+ \frac{C_{1}M_{q}\alpha\tau}{\alpha(1-q)}a^{\alpha(1-q)}
\int_{0}^{\delta}\theta^{1-q}\zeta_{\alpha}(\theta)d\theta\\
&+ \frac{C_{1}M_{q}\alpha\tau\Gamma(2-q)}{\Gamma(1+\alpha(1-q))}
\frac{1}{\alpha(1-q)}g^{\alpha(1-q)}.
\end{split}
\end{equation*}
Therefore, $\Pi(t)=\lbrace (Pu)(t)\vert u\in \Sigma\rbrace$ is relatively compact
in $X_{q}$ for all $t\in (0, a]$ and, since it is compact at $t=0$,
we have relatively compactness in $X_{q}$ for all $t\in J$.

Next, let us prove that $\Pi=P\Sigma$ is equicontinuous. For $g\in [0, a)$,
\begin{equation*}
\begin{split}
\Vert(Pu)(g)-(Pu)(0)\Vert_{q}
\leq & C_{2}\Vert v_{0}\Vert_{q}\Vert S_{\alpha}(g)L-I\Vert_{q}\\
&+ C_{1}C_{2}M_{0}\Vert L\Vert \left[(k_{2}+\Vert u_{0}\Vert_{q})\frac{g^{1-\alpha}}{\Gamma(2-\alpha)}\right]\\
&+ \frac{\alpha C_{1}M_{q}\Gamma(2-q)}{\Gamma(1+\alpha(1-q))}\frac{\tau}{\alpha(1-q)}g^{\alpha(1-q)},
\end{split}
\end{equation*}
and for $0<s<t_{1}<t_{2}\leq a$,
$\Vert(Pu)(t_{1})-(Pu)(t_{2})\Vert_{q}\leq I_{1}+I_{2}+I_{3}+I_{4}+I_{5}+I_{6}$, where
\begin{eqnarray*}
I_{1}&=&C_{2}\Vert L\Vert \left[\Vert v_{0}\Vert_{q}+(k_{2}+\Vert u_{0}
\Vert_{q})\frac{t_{1}^{1-\alpha}}{\Gamma(2-\alpha)}\right]
\Vert S_{\alpha}(t_{1})-S_{\alpha}(t_{2})\Vert_{q},\\
I_{2}&=&C_{1}C_{2}M_{0}\Vert L\Vert \left[(k_{2}+\Vert u_{0}\Vert_{q})
\frac{(t_{2}-t_{1})^{1-\alpha}+t_{1}^{1-\alpha}-t_{2}^{1-\alpha}}{\Gamma(2-\alpha)}\right],\\
I_{3}&=&C_{1}C_{2}M_{0}\Vert L\Vert \left[(k_{2}+\Vert u_{0}\Vert_{q})
\frac{(t_{2}-t_{1})^{1-\alpha}}{\Gamma(2-\alpha)}\right],\\
I_{4}&=&\frac{\alpha C_{1}M_{q}\Gamma(2-q)}{\Gamma(1+\alpha(1-q))}
\Vert f\Vert_{C(J, X)}\sqrt{\frac{1}{2\alpha-1}}t_{1}^{\alpha-\frac{1}{2}}\left(
\int_{0}^{a}\vert (t_{1}-s)^{-q\alpha}-(t_{2}-s)^{-q\alpha}\vert^{2} ds\right)^{\frac{1}{2}},\\
I_{5}&=&\frac{\alpha C_{1}M_{q}\Gamma(2-q)}{\Gamma(1+\alpha(1-q))}\Vert f\Vert_{C(J, X)}\left(
\int_{0}^{a}\vert(t_{1}-s)^{\alpha-1}-(t_{2}-s)^{\alpha-1}\vert^{2} ds\right)^{\frac{1}{2}}\\
&&\times\sqrt{\frac{1}{1-2q\alpha}}\biggl(t_{2}^{1-2q\alpha}-(t_{2}-t_{1})^{1-2q\alpha}\biggr)^{\frac{1}{2}},\\
I_{6}&=&\frac{\alpha C_{1}M_{q}\Gamma(2-q)}{\Gamma(1+\alpha(1-q))}
\Vert f\Vert_{C(J, X)}\frac{1}{\alpha(1-q)}(t_{2}-t_{1})^{\alpha(1-q)}.
\end{eqnarray*}
Now, we have to verify that $I_{j}$, $j = 1, \ldots, 6$, tend to $0$
independently of $u\in \Sigma$ when $t_{2}\rightarrow t_{1}$. Let $u\in \Sigma$.
By Lemma~\ref{Lemma 2.2} (c) and (f), we deduce that
$\lim_{t_{2}\rightarrow t_{1}}I_{1}=0$ and $\lim_{t_{2}\rightarrow t_{1}}I_{4}=0$.
Moreover, using the fact that $\vert(t_{1}-s)^{\alpha-1}-(t_{2}-s)^{\alpha-1}\vert\rightarrow 0$
as $t_{2}\rightarrow t_{1}$, we obtain from Lemma~\ref{Lemma 2.3} that
$$
\int_{0}^{a}\vert(t_{1}-s)^{\alpha-1}-(t_{2}-s)^{\alpha-1}\vert^{2}ds
\rightarrow 0 ~\text{as} ~ t_{2}\rightarrow t_{1}.
$$
Thus, $\lim_{t_{2}\rightarrow t_{1}}I_{5}=0$ since $q\alpha<\frac{1}{2}$.
Also, it is clear that $\lim_{t_{2}\rightarrow t_{1}}I_{2}=I_{3}=I_{6}=0$.
In summary, we have proven that $P\Sigma$ is relatively compact
for $t\in J$ and $\Pi(t)=\lbrace Pu\vert u\in \Sigma\rbrace$
is a family of equicontinuous functions. Hence,
by the Arzela--Ascoli theorem, $P$ is compact.
\end{proof}

\begin{proof}[Proof of Theorem~\ref{Theorem 3.1}]
We shall prove that the operator $P$ has a fixed point in $\Omega_{q}$.
According to Leray--Schauder fixed point theory
(and from Lemmas~\ref{Lemma 3.1}--\ref{Lemma 3.3}),
it suffices to show that the set
$\Delta=\lbrace u\in \Omega_{q}\vert u=\beta Pu, \beta\in [0, 1]\rbrace$
is a bounded subset of $\Omega_{q}$. Let $u\in \Delta$. Then,
\begin{equation*}
\begin{split}
\Vert u(t)\Vert_{q}
&=\Vert\beta(Pu)(t)\Vert_{q}\\
&\leq \left\Vert S_{\alpha}(t)LM^{-1}\left[v_{0}+\frac{1}{\Gamma(1-\alpha)}
\int_{0}^{t}\frac{[u_{0}+h(u(s))]}{(t-s)^{\alpha}}ds\right]\right\Vert_{q}\\
&\quad +\int_{0}^{t}(t-s)^{\alpha-1}\Vert T_{\alpha}(t-s)f(s, W(s))\Vert_{q}ds\\
&\leq C_{1}C_{2}M_{0}\Vert L\Vert\left[\Vert v_{0}\Vert_{q}
+\frac{a^{1-\alpha}}{\Gamma(2-\alpha)}(k_{2}+\Vert u_{0}\Vert_{q})\right]\\
&\quad +\int_{0}^{t}(t-s)^{\alpha-1}\Vert A^{q}T_{\alpha}(t-s)\Vert \Vert f(s, W(s))\Vert ds\\
&\leq C_{1}C_{2}M_{0}\Vert L\Vert\left[\Vert v_{0}\Vert_{q}
+\frac{a^{1-\alpha}}{\Gamma(2-\alpha)}(k_{2}+\Vert u_{0}\Vert_{q})\right]\\
&\quad +\frac{a_{f}\alpha C_{1}M_{q}\Gamma(2-q)}{\Gamma(1+\alpha(1-q))}
\int_{0}^{t}(t-s)^{-q\alpha+\alpha-1}(1+r\Vert u\Vert_{q})ds\\
&\leq C_{1}C_{2}M_{0}\Vert L\Vert\left[\Vert v_{0}\Vert_{q}
+\frac{a^{1-\alpha}}{\Gamma(2-\alpha)}(k_{2}+\Vert u_{0}\Vert_{q})\right]\\
&\quad +\frac{a_{f}\alpha C_{1}M_{q}\Gamma(2-q)}{\Gamma(1+\alpha(1-q))}
\frac{a^{\alpha(1-q)}}{\alpha(1-q)}+\frac{a_{f}\alpha rC_{1}M_{q}
\Gamma(2-q)}{\Gamma(1+\alpha(1-q))}\int_{0}^{t}(t-s)^{-q\alpha+\alpha-1}\Vert u\Vert_{q}ds.
\end{split}
\end{equation*}
Based on the well known singular version of Gronwall inequality, we can deduce
that there exists a constant $R>0$ such that $\Vert u\Vert_{\infty}\leq R$.
Thus, $\Delta$ is a bounded subset of $\Omega_{q}$. By Leray--Schauder fixed
point theory, $P$ has a fixed point in $\Omega_{q}$. Consequently,
system \eqref{eq:1.1}--\eqref{eq:1.2} has at least one mild solution $u$ on $J$.
\end{proof}

\begin{theorem}
\label{Theorem 3.2}
Mild solution $u(\cdot)$ of system \eqref{eq:1.1}--\eqref{eq:1.2} is unique.
\end{theorem}

\begin{proof}
Let $u^{*}(\cdot)$ be another mild solution of system \eqref{eq:1.1}--\eqref{eq:1.2}
with Sobolev--fractional nonlocal initial value
$M^{-1}\left[v_{0}+\frac{1}{\Gamma(1-\alpha)}\int_{0}^{t}\frac{[u_{0}+h(u(s))]}{(t-s)^{\alpha}}ds\right]$.
It is not difficult to verify that there exists a constant $\rho>0$ such that
$\Vert u\Vert_{q}, \Vert u^{*}\Vert_{q}\leq\rho$. From
\begin{multline*}
\Vert u(t)-u^{*}(t)\Vert_{q}
\leq \left\Vert S_{\alpha}(t)LM^{-1}
\left\lbrace[v_{0}-v^{*}_{0}]+\frac{1}{\Gamma(1-\alpha)}\int_{0}^{t}
\frac{[u_{0}-u^{*}_{0}]+[h(u)-h(u^{*})]}{(t-s)^{\alpha}}ds\right\rbrace\right\Vert_{q}\\
+\int_{0}^{t}(t-s)^{\alpha-1}\Vert T_{\alpha}(t-s)[f(s, W(s))-f(s, W^{*}(s))]\Vert_{q}ds,
\end{multline*}
we get
\begin{multline*}
\Vert u(t)-u^{*}\Vert_{q}
\leq C_{1}C_{2}M_{0}\Vert L\Vert\left\lbrace\Vert v_{0}
-v^{*}_{0}\Vert_{q}+\frac{1}{\Gamma(1-\alpha)}\int_{0}^{t}\frac{\Vert u_{0}-u^{*}_{0}\Vert_{q}
+k_{1}\Vert u(s)-u^{*}(s)\Vert_{q}}{(t-s)^{\alpha}}ds\right\rbrace\\
+ L_{f}(\rho)\sum\limits_{i=1}^{r}m_{i}(t)\frac{\alpha C_{1}M_{q}\Gamma(2-q)}{\Gamma(1+\alpha(1-q))}
\int_{0}^{t}(t-s)^{-q\alpha+\alpha-1}\Vert u(s)-u^{*}(s)\Vert_{q}ds.
\end{multline*}
Again, by the singular version of Gronwall's inequality,
there exists a constant $R^{*}>0$ such that
$$
\Vert u(t)-u^{*}(t)\Vert_{q}\leq C_{1}C_{2}M_{0}
\Vert L\Vert R^{*}\Vert u_{0}-u_{0}^{*}\Vert_{q},
$$
which gives the uniqueness of $u$. Thus, system \eqref{eq:1.1}--\eqref{eq:1.2}
has a unique mild solution on $J$.
\end{proof}


\section{Optimal multi-integral controls}
\label{sec:4}

Let $Z$ be another separable reflexive Banach space from which the controls
$\mathfrak{u}_{1},\ldots,\mathfrak{u}_{k}$ take their values. We denote by
$V_{f}(Z)$ a class of nonempty closed and convex subsets of $Z$. The multifunction
$\omega: J\rightarrow V_{f}(Z)$ is measurable, $\omega(\cdot)\subset \Lambda$,
where $\Lambda$ is a bounded set of $Z$. The admissible control set
is $U_{ad}=S^{p}_{\omega}=\lbrace \mathfrak{u}_{j}\in L^{p}(\Lambda)
\vert \mathfrak{u}_{j}(t)\in \omega(t)~ a.e. \rbrace$,
$j=\overline{1,k}$, $1<p<\infty$. Then, $U_{ad}\neq \emptyset$ \cite{AMA.39}.

Consider the following Sobolev type fractional nonlocal multi-integral-controlled system:
\begin{equation}
\label{eq:4.1}
^CD^{\alpha}_{t}[Lu(t)]=Eu(t)+f(t, W(t))+\int_{0}^{t}[\mathcal{B}_{1}\mathfrak{u}_{1}(s)
+\cdots+\mathcal{B}_{k}\mathfrak{u}_{k}(s)]ds,
\end{equation}
\begin{equation}
\label{eq:4.2}
^{L}D^{1-\alpha}_{t}[Mu(0)]=u_{0}+h(u(t)).
\end{equation}
Besides the sufficient conditions ($F_1$)--($F_5$) of the last section, we assume:
\begin{itemize}
\item[($F_6$)] $ \mathcal{B}_{j}\in L^{\infty}(J, L(Z, X_{q}))$, which implies that
$\mathcal{B}_{j}\mathfrak{u}_{j}\in L^{p}(J, X_{q})$ for all $\mathfrak{u}_{j}\in U_{ad}$.
\end{itemize}

\begin{corollary}
\label{Theorem 4.1}
In addition to assumptions of Theorem~\ref{Theorem 3.1},
suppose ($F_6$) holds. For every $\mathfrak{u}_{j}\in U_{ad}$
and $p\alpha(1-q)>1$, system \eqref{eq:4.1}--\eqref{eq:4.2}
has a mild solution corresponding to $\mathfrak{u}_{j}$ given by
\begin{multline*}
u^{\mathfrak{u}_{j}}(t)
= S_{\alpha}(t)LM^{-1}\left[v_{0}+\frac{1}{\Gamma(1-\alpha)}
\int_{0}^{t}\frac{[u_{0}+h(u(s))]}{(t-s)^{\alpha}}ds\right]\\
+\int_{0}^{t}(t-s)^{\alpha-1}T_{\alpha}(t-s)\left[f(s, W(s))
+\int_{0}^{s}[\mathcal{B}_{1}\mathfrak{u}_{1}(\eta)
+\cdots+\mathcal{B}_{k}\mathfrak{u}_{k}(\eta)]d\eta\right]ds.
\end{multline*}
\end{corollary}

\begin{proof}
Based on our existence result (Theorem~\ref{Theorem 3.1}), it is required to check
the term containing multi-integral controls. Let us consider
$$
\varphi(t)=\int_{0}^{t}(t-s)^{\alpha-1}T_{\alpha}(t-s)\left[
\int_{0}^{s}[\mathcal{B}_{1}\mathfrak{u}_{1}(\eta)
+\cdots+\mathcal{B}_{k}\mathfrak{u}_{k}(\eta)]d\eta\right]ds.
$$
Using Lemma~\ref{Lemma 2.2} (d) and H\"older inequality, we have
\begin{eqnarray*}
\Vert\varphi(t)\Vert_{q}&\leq&\left\Vert\int_{0}^{t}(t-s)^{\alpha-1}
T_{\alpha}(t-s)\int_{0}^{s}[\mathcal{B}_{1}\mathfrak{u}_{1}(\eta)
+\cdots+\mathcal{B}_{k}\mathfrak{u}_{k}(\eta)]d\eta ds\right\Vert_{q}\\
&\leq&\int_{0}^{t}(t-s)^{\alpha-1}\Vert A^{q}T_{\alpha}(t-s)
\Vert[\Vert\mathcal{B}_{1}\mathfrak{u}_{1}(s)\Vert a
+\cdots+\Vert\mathcal{B}_{k}\mathfrak{u}_{k}(s)\Vert a]ds\\
&\leq&\frac{\alpha aC_{1}M_{q}\Gamma(2-q)}{
\Gamma(1+\alpha(1-q))}\biggl[\Vert\mathcal{B}_{1}\Vert_{\infty}
\int_{0}^{t}(t-s)^{-q\alpha+\alpha-1}\Vert\mathfrak{u}_{1}(s)\Vert_{Z}ds\\
&&+\cdots+\Vert\mathcal{B}_{k}\Vert_{\infty}
\int_{0}^{t}(t-s)^{-q\alpha+\alpha-1}\Vert\mathfrak{u}_{k}(s)\Vert_{Z}ds\biggr]\\
&\leq&\frac{\alpha aC_{1}M_{q}\Gamma(2-q)}{\Gamma(1+\alpha(1-q))}\biggl[\Vert\mathcal{B}_{1}\Vert_{\infty}
\left(\int_{0}^{t}(t-s)^{\frac{p}{p-1}(-q\alpha+\alpha-1)}ds\right)^{\frac{p-1}{p}}\left(
\int_{0}^{t}\Vert\mathfrak{u}_{1}(s)\Vert^{p}_{Z}ds\right)^{\frac{1}{p}}\\
&&+\cdots+\Vert\mathcal{B}_{k}\Vert_{\infty}
\left(\int_{0}^{t}(t-s)^{\frac{p}{p-1}(-q\alpha+\alpha-1)}ds\right)^{\frac{p-1}{p}}\left(
\int_{0}^{t}\Vert\mathfrak{u}_{k}(s)\Vert^{p}_{Z}ds\right)^{\frac{1}{p}}\biggr]\\
&\leq&\frac{\alpha aC_{1}M_{q}\Gamma(2-q)}{\Gamma(1+\alpha(1-q))}\biggl[
\Vert\mathcal{B}_{1}\Vert_{\infty}\left(\frac{p-1}{p\alpha(1-q)-1}\right)^{\frac{p-1}{p}}
a^{\frac{p\alpha(1-q)-1}{p-1}}\Vert\mathfrak{u}_{1}\Vert_{L^{p}(J, Z)}\\
&&+\cdots+\Vert\mathcal{B}_{k}\Vert_{\infty}\left(\frac{p-1}{p\alpha(1-q)-1}\right)^{\frac{p-1}{p}}
a^{\frac{p\alpha(1-q)-1}{p-1}}\Vert\mathfrak{u}_{k}\Vert_{L^{p}(J, Z)}\biggr],
\end{eqnarray*}
where $\Vert\mathcal{B}_{1}\Vert_{\infty},\ldots,\Vert\mathcal{B}_{k}\Vert_{\infty}$
are the norm of operators $\mathcal{B}_{1},\ldots,\mathcal{B}_{k}$,
respectively, in the Banach space $L_{\infty}(J, L(Z, X_{q}))$. Thus,
$$
\left\Vert(t-s)^{\alpha-1}T_{\alpha}(t-s)\int_{0}^{s}[\mathcal{B}_{1}\mathfrak{u}_{1}(\eta)
+\cdots+\mathcal{B}_{k}\mathfrak{u}_{k}(\eta)]d\eta\right\Vert_{q}
$$
is Lebesgue integrable with respect to $s\in[0, t]$ for all $t\in J$.
It follows from Lemma~\ref{Lemma 2.4} that
$$
(t-s)^{\alpha-1}T_{\alpha}(t-s)\int_{0}^{s}[\mathcal{B}_{1}\mathfrak{u}_{1}(\eta)
+\cdots+\mathcal{B}_{k}\mathfrak{u}_{k}(\eta)]d\eta
$$
is a Bochner integral with respect to $s\in[0, t]$ for all $t\in J$.
Hence, $\varphi(\cdot)\in \Omega_{q}$.
The required result follows from Theorem~\ref{Theorem 3.1}.
\end{proof}

Furthermore, let us now assume
\begin{itemize}
\item[($F_7$)] The functional $\mathcal{L}: J\times X_{q}
\times Z^{k}\rightarrow\mathbb{R}\cup\lbrace\infty\rbrace$ is Borel measurable.
\item[($F_8$)] $\mathcal{L}(t,\cdot,\ldots,\cdot)$ is sequentially lower semicontinuous
on $X_{q}\times Z^{k}$ for almost all $t\in J$.
\item[($F_9$)]$\mathcal{L}(t, u,\cdot,\ldots,\cdot)$ is convex on $Z^{k}$
for each $u\in X_{q}$ and almost all $t\in J$.
\item[($F_{10}$)] There exist constants $d\geq 0$, $c_{1}, \ldots, c_{k}>0$,
such that $\psi$ is nonnegative and $\psi\in L^{1}(J, \mathbb{R})$ satisfies
$$
\mathcal{L}(t, u,\mathfrak{u}_{1},\ldots,\mathfrak{u}_{k})\geq\psi(t)
+d\Vert u\Vert_{q}+c_{1}\Vert\mathfrak{u}_{1}\Vert^{p}_{Z}
+\cdots+c_{k}\Vert\mathfrak{u}_{k}\Vert^{p}_{Z}.
$$
\end{itemize}
We consider the following Lagrange optimal control problem:
\begin{equation}
\label{LP}
\tag{$LP$}
\left\{
\begin{array}{ll}
\text{Find}~(u^{0}, \mathfrak{u}^{0}_{1},\ldots,
\mathfrak{u}^{0}_{k})\in C(J, X_{q})\times U^{k}_{ad} \\
\text{such that}~ \mathcal{J}(u^{0}, \mathfrak{u}^{0}_{1},\ldots,\mathfrak{u}^{0}_{k})
\leq\mathcal{J}(u^{\mathfrak{u}_{1},\ldots,\mathfrak{u}_{k}},
\mathfrak{u}_{1},\ldots,\mathfrak{u}_{k})~ \text{for all}~ \mathfrak{u}_{j}\in U_{ad},
\end{array}\right.
\end{equation}
where
$$
\mathcal{J}(u^{\mathfrak{u}_{1},\ldots,\mathfrak{u}_{k}}, \mathfrak{u}_{1},
\ldots,\mathfrak{u}_{k})=\int_{0}^{a}\mathcal{L}(t, u^{\mathfrak{u}_{1},
\ldots,\mathfrak{u}_{k}},\mathfrak{u}_{1}(t),\ldots,\mathfrak{u}_{k}(t))dt
$$
with $u^{\mathfrak{u}_{j}}$ denoting the mild solution of system \eqref{eq:4.1}--\eqref{eq:4.2}
corresponding to the multi-integral controls $\mathfrak{u}_{j}\in U_{ad}$.
The following lemma is used to obtain existence
of a fractional optimal multi-integral control (Theorem~\ref{Theorem 4.2}).

\begin{lemma}
\label{Lemma 4.1}
Operators $\Upsilon_{j}: L^{p}(J, Z)\rightarrow \Omega_{q}$ given by
$$
\begin{cases}
(\Upsilon_{1}\mathfrak{u}_{1})(\cdot)=\int_{0}^{\cdot}
\int_{0}^{s}T_{\alpha}(\cdot-s)\mathcal{B}_{1}\mathfrak{u}_{1}(\eta)d\eta ds,\\
\qquad \vdots\\
(\Upsilon_{1}\mathfrak{u}_{k})(\cdot)=\int_{0}^{\cdot}
\int_{0}^{s}T_{\alpha}(\cdot-s)\mathcal{B}_{k}\mathfrak{u}_{k}(\eta)d\eta ds,
\end{cases}
$$
where $p\alpha(1-q)>1$ and $j=\overline{1,k}$, are strongly continuous.
\end{lemma}

\begin{proof}
Suppose that $\lbrace\mathfrak{u}^{n}_{j}\rbrace_{j=\overline{1,k}}\subseteq L^{p}(J, Z)$ are bounded.
Define $\Theta_{j, n}(t)=(\Upsilon_{j}\mathfrak{u}_{j}^{n})(t)$, $t\in J$.
Similarly to the proof of Corollary~\ref{Theorem 4.1}, we can conclude that for any fixed
$t\in J$ and $p\alpha(1-q)>1$, $\Vert\Theta_{j, n}(t)\Vert_{q}, j=\overline{1,k}$, are bounded.
By Lemma~\ref{Lemma 2.2}, it is easy to verify that $\Theta_{j, n}(t)$, $j=\overline{1,k}$,
are compact in $X_{q}$ and are also equicontinuous. According to the Ascoli--Arzela theorem,
$\lbrace\Theta_{j, n}(t)\rbrace$ are relatively compact in $\Omega_{q}$. Clearly,
$\Upsilon_{j}$, $j=\overline{1,k}$, are linear and continuous. Hence,
$\Upsilon_{j}$ are strongly continuous operators (see \cite[p. 597]{AMA.39}).
\end{proof}

Now we are in position to give the following result on existence
of optimal multi-integral controls for the Lagrange problem \eqref{LP}.

\begin{theorem}
\label{Theorem 4.2}
If the assumptions ($F_{1}$)--($F_{10}$) hold, then the Lagrange problem
\eqref{LP} admits at least one optimal multi-integral pair.
\end{theorem}

\begin{proof}
Assume that $\inf\lbrace\mathcal{J}(u^{\mathfrak{u}_{1},\ldots,\mathfrak{u}_{k}},
\mathfrak{u}_{1},\ldots,\mathfrak{u}_{k})\vert u^{\mathfrak{u}_{j}}
\in U_{ad}\rbrace=\epsilon <+\infty$. Using assumptions ($F_{7}$)--($F_{10}$),
we have $\epsilon>-\infty$. By definition of infimum, there exists a minimizing feasible multi-pair
$\lbrace(u^{m}, \mathfrak{u}^{m}_{1},\ldots,\mathfrak{u}^{m}_{k})\rbrace\subset\mathcal{U}_{ad}$ sequence,
where $\mathcal{U}_{ad}=\lbrace (u, \mathfrak{u}_{1},\ldots,\mathfrak{u}_{k})\vert u$ is a mild solution
of system \eqref{eq:4.1}--\eqref{eq:4.2} corresponding to $\mathfrak{u}_{1},\ldots,\mathfrak{u}_{k}\in U_{ad}\rbrace$,
such that $\mathcal{J}(u^{m}, \mathfrak{u}^{m}_{1},\ldots,\mathfrak{u}^{m}_{k})\rightarrow\epsilon$ as $m\rightarrow+\infty$.
Since $\lbrace(\mathfrak{u}^{m}_{1},\ldots,\mathfrak{u}^{m}_{k})\rbrace\subseteq U_{ad}$,
$m=1,2,\ldots, \lbrace(\mathfrak{u}^{m}_{1},\ldots,\mathfrak{u}^{m}_{k})\rbrace$
is bounded in $L^{p}(J, Z)$ and there exists a subsequence, still denoted by
$\lbrace(\mathfrak{u}^{m}_{1},\ldots,\mathfrak{u}^{m}_{k})\rbrace$,
$\mathfrak{u}^{0}_{1},\ldots,\mathfrak{u}^{0}_{k}\in L^{p}(J, Z)$, such that
$$
\left(\mathfrak{u}^{m}_{1},\ldots,\mathfrak{u}^{m}_{k}\right)
\stackrel{\text{weakly}}{\longrightarrow}\left(
\mathfrak{u}^{0}_{1},\ldots,\mathfrak{u}^{0}_{k}\right)
$$
in $L^{p}(J, Z)$. Since $U_{ad}$ is closed and convex,
by Marzur lemma $\mathfrak{u}^{0}_{1},\ldots,\mathfrak{u}^{0}_{k}\in U_{ad}$.
Suppose $u^{m}(u^{0})$ is the mild solution of system \eqref{eq:4.1}--\eqref{eq:4.2}
corresponding to $\mathfrak{u}^{m}_{1}(\mathfrak{u}^{0}_{1}),\ldots$,
$\mathfrak{u}^{m}_{k}(\mathfrak{u}^{0}_{k})$. Functions $u^{m}$ and $u^{0}$ satisfy,
respectively, the following integral equations:
\begin{multline*}
u^{m}(t) = S_{\alpha}(t)LM^{-1}\left[v_{0}+\frac{1}{\Gamma(1-\alpha)}
\int_{0}^{t}\frac{[u_{0}+h(u^{m}(s))]}{(t-s)^{\alpha}}ds\right]\\
+\int_{0}^{t}(t-s)^{\alpha-1}T_{\alpha}(t-s)\left[f(s, W^{m}(s))+
\int_{0}^{s}[\mathcal{B}_{1}\mathfrak{u}^{m}_{1}(\eta)
+\cdots+\mathcal{B}_{k}\mathfrak{u}^{m}_{k}(\eta)]d\eta\right]ds,
\end{multline*}
\begin{multline*}
u^{0}(t) = S_{\alpha}(t)LM^{-1}\left[v_{0}+\frac{1}{\Gamma(1-\alpha)}
\int_{0}^{t}\frac{[u_{0}+h(u^{0}(s))]}{(t-s)^{\alpha}}ds\right]\\
+\int_{0}^{t}(t-s)^{\alpha-1}T_{\alpha}(t-s)\left[f(s, W^{0}(s))
+\int_{0}^{s}[\mathcal{B}_{1}\mathfrak{u}^{0}_{1}(\eta)
+\cdots+\mathcal{B}_{k}\mathfrak{u}^{0}_{k}(\eta)]d\eta\right]ds.
\end{multline*}
It follows from the boundedness of $\lbrace\mathfrak{u}^{m}_{1}\rbrace,
\ldots,\lbrace\mathfrak{u}^{m}_{k}\rbrace, \lbrace\mathfrak{u}^{0}_{1}
\rbrace,\ldots,\lbrace\mathfrak{u}^{0}_{k}\rbrace$ and Theorem~\ref{Theorem 3.1}
that there exists a positive number $\rho$ such that
$\Vert u^{m}\Vert_{\infty}, \Vert u^{0}\Vert_{\infty}\leq \rho$.
For $t\in J$, we have
$$
\Vert u^{m}(t)-u^{0}(t)\Vert_{q}\leq \Vert\xi^{(1)}_{m}(t)\Vert_{q}
+\Vert\xi^{(2)}_{m}(t)\Vert_{q}+\Vert\xi^{(3)}_{m}(t)\Vert_{q}
+\cdots+\Vert\xi^{(k+2)}_{m}(t)\Vert_{q},
$$
where
\begin{eqnarray*}
\xi^{(1)}_{m}(t)&=&S_{\alpha}(t)LM^{-1}\frac{1}{\Gamma(1-\alpha)}
\int_{0}^{t}\frac{[h(u^{m}(s))-h(u^{0}(s))]}{(t-s)^{\alpha}}ds,\\
\xi^{(2)}_{m}(t)&=&\int_{0}^{t}(t-s)^{\alpha-1}T_{\alpha}(t-s)[f(s, W^{m}(s))-f(s, W^{0}(s))]ds,\\
\xi^{(3)}_{m}(t)&=&\int_{0}^{t}(t-s)^{\alpha-1}T_{\alpha}(t-s)
\int_{0}^{s}\mathcal{B}_{1}[\mathfrak{u}^{m}_{1}(\eta)-\mathfrak{u}^{0}_{1}(\eta)]d\eta ds,\\
&\vdots&\\
\xi^{(k+2)}_{m}(t)&=&\int_{0}^{t}(t-s)^{\alpha-1}T_{\alpha}(t-s)
\int_{0}^{s}\mathcal{B}_{k}[\mathfrak{u}^{m}_{k}(\eta)-\mathfrak{u}^{0}_{k}(\eta)]d\eta ds.
\end{eqnarray*}
The assumption ($F_{5}$) gives
$$
\Vert\xi^{(1)}_{m}(t)\Vert_{q}\leq C_{1}C_{2}M_{0}k_{1}\Vert L\Vert
\frac{a^{1-\alpha}}{\Gamma(2-\alpha)}\Vert u^{m}-u^{0}\Vert_{q}.
$$
Using Lemma~\ref{Lemma 2.2} (d) and ($F_{3}$),
$$
\Vert\xi^{(2)}_{m}(t)\Vert_{q}\leq L_{f}(\rho)\sum\limits_{i=1}^{r}
m_{i}(t)\frac{\alpha C_{1}M_{q}\Gamma(2-q)}{\Gamma(1+\alpha(1-q))}
\int_{0}^{t}(t-s)^{-q\alpha+\alpha-1}\Vert u^{m}(s)-u^{0}(s)\Vert_{q}ds.
$$
From Lemma~\ref{Lemma 4.1}, we get
$$
\xi^{(j+2)}_{m}(t)\stackrel{\text{strongly}}{\longrightarrow}0\
\text{ in }\  X_{q}~\text{as}~m\rightarrow\infty,
\quad j=\overline{1,k}.
$$
Thus,
\begin{multline*}
\Vert u^{m}(t)-u^{0}(t)\Vert_{q}
\leq \sum\limits_{j=1}^{k}\Vert\xi^{(j+2)}_{m}(t)\Vert_{q}
+C_{1}C_{2}M_{0}k_{1}\Vert L\Vert\frac{a^{1-\alpha}}{\Gamma(2-\alpha)}\Vert u^{m}-u^{0}\Vert_{q}\\
+ L_{f}(\rho)\sum\limits_{i=1}^{r}m_{i}(t)\frac{\alpha C_{1}M_{q}\Gamma(2-q)}{\Gamma(1+\alpha(1-q))}
\int_{0}^{t}(t-s)^{-q\alpha+\alpha-1}\Vert u^{m}(s)-u^{0}(s)\Vert_{q}ds.
\end{multline*}
By virtue of the singular version of Gronwall's inequality, there exists $M_{*}>0$ such that
$$
\Vert u^{m}(t)-u^{0}(t)\Vert_{q}\leq M_{*}\sum\limits_{j=1}^{k}\Vert\xi^{(j+2)}_{m}(t)\Vert_{q},
$$
which yields that
$$
u^{m}\rightarrow u^{0}~ \text{in}~ C(J, X_{q})~ \text{as}~ m\rightarrow\infty.
$$
Because $C(J, X_{q})\hookrightarrow L^{1}(J, X_{q})$, using the assumptions
($F_{7}$)--($F_{10}$) and Balder's theorem, we obtain that
\begin{equation*}
\begin{split}
\epsilon&=\lim\limits_{m\rightarrow\infty}\int_{0}^{a}\mathcal{L}(t, u^{m}(t),
\mathfrak{u}^{m}_{1}(t),\ldots,\mathfrak{u}^{m}_{k}(t))dt\\
&\geq \int_{0}^{a}\mathcal{L}(t, u^{0}(t),\mathfrak{u}^{0}_{1}(t),
\ldots,\mathfrak{u}^{0}_{k}(t))dt\\
&=\mathcal{J}(u^{0}, \mathfrak{u}^{0}_{1},\ldots,\mathfrak{u}^{0}_{k})\\
&\geq \epsilon.
\end{split}
\end{equation*}
This shows that $\mathcal{J}$ attains its minimum at
$\mathfrak{u}^{0}_{1},\ldots,\mathfrak{u}^{0}_{k}\in U_{ad}$.
\end{proof}


\section{An example}
\label{sec:5}

Consider the following fractional nonlocal multi-controlled system of Sobolev type:
\begin{equation}
\label{eq:5.1}
\frac{\partial^{\alpha}}{\partial t^{\alpha}}\biggl[u(t, x)- u_{xx}(t, x)\biggr]
+\frac{\partial^{2}}{\partial x^{2}}u(t, x)=\int_{0}^{t}[\mathfrak{u}_{1}(s, x)
+\cdots+\mathfrak{u}_{k}(s, x)]ds+F(t, D^{r}_{x}u(x, t)),
\end{equation}
\begin{equation}
\label{eq:5.2}
u(0, x)=\frac{\partial^{2}}{\partial x^{2}}\left[v_{0}(x)+\sum\limits_{\eta=1}^{m}
\frac{c_{\eta}}{\Gamma(1-\alpha)}\int_{0}^{t_{\eta}}\frac{u_{0}(x)
+u(s_{\eta}, x)}{(t_{\eta}-s_{\eta})^{\alpha}}ds_{\eta}\right],~x\in[0, \pi],
\end{equation}
\begin{equation}
\label{eq:5.3}
u(t, 0)=u(t, \pi)=0, \quad 0<t \leq 1,
\end{equation}
where $0<\alpha\leq1$, $0< t_{1}<\cdots<t_{m}< 1$ and $c_{\eta}$ are positive constants,
$\eta=1,\ldots,m$; the functions $u(t)(x)=u(t, x), f (t, \cdot)=F(t, \cdot), W(t)(x)
=D^{r}_{x}u(x, t)$ and $h(u(t))(x)=\sum_{\eta=1}^{m}c_{\eta}u(t_{\eta}, x)$.
Let us take $\mathcal{B}_{j}\mathfrak{u}_{j}(t)(x)=\mathfrak{u}_{j}(t, x)$,
$j=\overline{1,k}$, and the operator $D^{r}_{x}$ as follows:
$$
D^{r}_{x}u(x, t)=\left(\partial_{x}u(x, t),
\partial^{2}_{x}u(x, t),\ldots,\partial^{r}_{x}u(x, t)\right).
$$
Let $X=Y=Z= L^{2}[0, \pi]$. Define the operators $L$, $E$, and $M$ on domains and ranges
contained in $L^{2}[0, \pi]$ by $Lw=w-w^{\prime\prime}$, $Ew=-w^{\prime\prime}$
and $M^{-1}w=w^{\prime\prime}$, where the domains $D(L)$, $D(E)$ and $D(M)$ are given by
$$
\lbrace w\in X: w, w^{\prime} \  \text{ are absolutely continuous},\
w^{\prime\prime}\in X, w(0)=w(\pi)=0\rbrace.
$$
Then $L$ and $E$ can be written, respectively, as
$$
Lw=\sum\limits_{n=1}^{\infty}(1+n^{2})(w, w_{n})w_{n}\ \text{ and }\
Ew=\sum\limits_{n=1}^{\infty}-n^{2}(w, w_{n})w_{n},
$$
where $w_{n}(t)=(\sqrt{2/\pi})\sin nt$, $n=1,2,\ldots$,
is the orthogonal set of eigenfunctions of $E$.
Furthermore, for any $w\in X$, we have
$$
L^{-1}w=\sum\limits_{n=1}^{\infty}\frac{1}{1+n^{2}}(w, w_{n})w_{n},
\quad EL^{-1}w=\sum\limits_{n=1}^{\infty}\frac{-n^{2}}{1+n^{2}}(w, w_{n})w_{n},
$$
and
$$
Q(t)x=\sum\limits_{n=1}^{\infty}
\exp\left(\frac{-n^{2}t}{1+n^{2}}\right)(w, w_{n})w_{n}.
$$
It is easy to see that $L^{-1}$ is compact, bounded, with
$\Vert L^{-1}\Vert \leq1$, and $A=EL^{-1}$ generates the above
strongly continuous semigroup $Q(t)$ on $L^{2}[0, \pi]$
with $\Vert Q(t)\Vert\leq e^{-t}\leq 1$.
If $\mathcal{B}_{j}=0$, $j=\overline{1,k}$, then, with the above choices, system
\eqref{eq:5.1}--\eqref{eq:5.3} can be written in the form
\eqref{eq:1.1}--\eqref{eq:1.2}. Therefore, Theorems~\ref{Theorem 3.1}
and \ref{Theorem 3.2} can be applied to guarantee existence and uniqueness
of a mild solution to \eqref{eq:5.1}--\eqref{eq:5.3}.

Let the admissible control set be
$$
U_{ad}=\left\lbrace \mathfrak{u}_{j}\in Z ~ \vert ~
\sum_{j=1}^{k}\int_{0}^{t}\Vert\mathfrak{u}_{j}(s, x)\Vert_{L^{2}([0, 1], Z)}ds
\leq 1\right\rbrace.
$$
Choose $\alpha=\frac{4}{5}, p=2$ and $q=\frac{1}{4}$. Find the controls
$\mathfrak{u}_{1}(t, x),\ldots,\mathfrak{u}_{k}(t, x)$ that minimize the functional
$$
\mathcal{J}(u, \mathfrak{u}_{1},\ldots,\mathfrak{u}_{k})
=\int_{0}^{1}\int_{0}^{\pi}\vert u(t, x)\vert^{2}dx dt
+\sum\limits_{j=1}^{k}\int_{0}^{1}\int_{0}^{t}
\int_{0}^{\pi}\vert \mathfrak{u}_{j}(s, x)\vert^{2} dx ds dt
$$
subject to system \eqref{eq:5.1}--\eqref{eq:5.3}.
If $\mathcal{B}_{j}\mathfrak{u}_{j}(t)(x)=\mathfrak{u}_{j}(t, x)$,
$j=\overline{1,k}$, then system \eqref{eq:5.1}--\eqref{eq:5.3}
can be transformed into \eqref{eq:4.1}--\eqref{eq:4.2} with the cost function
$$
\mathcal{J}(\mathfrak{u}_{1},\ldots,\mathfrak{u}_{k})
=\int_{0}^{1}\left[\Vert u(t)\Vert^{2}
+\int_{0}^{t}\lbrace\Vert \mathfrak{u}_{1}(s)\Vert^{2}_{Z}
+\cdots+\Vert \mathfrak{u}_{k}(s)\Vert^{2}_{Z}\rbrace ds\right]dt.
$$
We can check that $\alpha q=\frac{4}{5}\times\frac{1}{4}
=\frac{1}{5}<1$ and $p\alpha(1-q)=2\frac{4}{5}\frac{3}{4}=\frac{6}{5}>1$.
Then all assumptions of Theorem~\ref{Theorem 4.2} are satisfied
and we conclude that the optimal control problem has an optimal pair.


\section*{Acknowledgments}

This work was partially supported by CIDMA and FCT
project PEst-OE/MAT/UI4106/2014. The first author
is very grateful to CIDMA and the Department of Mathematics
of University of Aveiro, for the hospitality and
the excellent working conditions.




\small

\begin{thebibliography}{99}

\bibitem{AMA.1}
D. Baleanu, K. Diethelm, E. Scalas\ and\ J. J. Trujillo,
{\it Fractional calculus}, Series on Complexity, Nonlinearity and Chaos, 3,
World Scientific Publishing Co. Pte. Ltd., Hackensack, NJ, 2012.

\bibitem{AMA.2}
R. Hilfer,
{\it Applications of fractional calculus in physics},
World Sci. Publishing, River Edge, NJ, 2000.

\bibitem{AMA.3}
A. A. Kilbas, H. M. Srivastava\ and\ J. J. Trujillo,
{\it Theory and applications of fractional differential equations},
North-Holland Mathematics Studies, 204, Elsevier, Amsterdam, 2006.

\bibitem{AMA.4}
I. Podlubny,
{\it Fractional differential equations},
Mathematics in Science and Engineering, 198,
Academic Press, San Diego, CA, 1999.

\bibitem{AMA.5}
V. Lakshmikantham, S. Leela\ and\ J. Vasundhara Devi,
{\it Theory of fractional dynamic systems},
Cambridge Scientific Publishers, Cambridge, 2009.

\bibitem{AMA.6}
A. B. Malinowska\ and\ D. F. M. Torres,
{\it Introduction to the fractional calculus of variations},
Imp. Coll. Press, London, 2012.

\bibitem{AMA.7}
K. S. Miller\ and\ B. Ross,
{\it An introduction to the fractional calculus and fractional differential equations},
A Wiley-Interscience Publication, Wiley, New York, 1993.

\bibitem{AMA.8}
J. Sabatier, O. P. Agrawal\ and\ J. A. Tenreiro Machado,
{\it Advances in fractional calculus}, Springer, Dordrecht, 2007.

\bibitem{AMA.9}
S. G. Samko, A. A. Kilbas\ and\ O. I. Marichev,
{\it Fractional integrals and derivatives},
translated from the 1987 Russian original,
Gordon and Breach, Yverdon, 1993.

\bibitem{AMA.10}
B. Ahmad\ and\ J. J. Nieto,
Existence of solutions for anti-periodic boundary value problems
involving fractional differential equations via Leray-Schauder degree theory,
Topol. Methods Nonlinear Anal. {\bf 35} (2010), no.~2, 295--304.

\bibitem{AMA.11}
M. R. Sidi Ammi, E. H. El Kinani\ and\ D. F. M. Torres,
Existence and uniqueness of solutions to functional integro-differential fractional equations,
Electron. J. Differential Equations {\bf 2012} (2012), no.~103, 9~pp.
{\tt arXiv:1206.3996}

\bibitem{AMA.12}
M. Kirane, A. Kadem\ and\ A. Debbouche,
Blowing-up solutions to two-times fractional differential equations,
Math. Nachr. {\bf 286} (2013) no.~17-18, 1797--1804.

\bibitem{AMA.13}
M. D. Ortigueira,
On the initial conditions in continuous time fractional linear systems,
Signal Process. {\bf 83} (2003), 2301--2309.

\bibitem{AMA.14}
F. Mainardi,
Fractional calculus: some basic problems in continuum and statistical mechanics,
in {\it Fractals and fractional calculus in continuum mechanics (Udine, 1996)}, 291--348,
CISM Courses and Lectures, 378, Springer, Vienna, 1997.

\bibitem{AMA.15}
G. M. Mophou,
Weighted pseudo almost automorphic mild solutions to semilinear fractional differential equations,
Appl. Math. Comput. {\bf 217} (2011), no.~19, 7579--7587.

\bibitem{AMA.16}
R. Sakthivel, N. I. Mahmudov\ and\ J. J. Nieto,
Controllability for a class of fractional-order neutral evolution control systems,
Appl. Math. Comput. {\bf 218} (2012), no.~20, 10334--10340.

\bibitem{AMA.17}
L. Zhang, B. Ahmad, G. Wang\ and\ R. P. Agarwal,
Nonlinear fractional integro-differential equations on unbounded domains in a Banach space,
J. Comput. Appl. Math. {\bf 249} (2013), 51--56.

\bibitem{MR2897776}
E. Bazhlekova,
Existence and uniqueness results for a fractional evolution equation in Hilbert space,
Fract. Calc. Appl. Anal. {\bf 15} (2012), no.~2, 232--243.

\bibitem{AMA.18}
M. Benchohra, E. P. Gatsori\ and\ S. K. Ntouyas,
Controllability results for semilinear evolution inclusions with nonlocal conditions,
J. Optim. Theory Appl. {\bf 118} (2003), no.~3, 493--513.

\bibitem{AMA.19}
A. Debbouche\ and\ D. Baleanu,
Controllability of fractional evolution nonlocal impulsive
quasilinear delay integro-differential systems,
Comput. Math. Appl. {\bf 62} (2011), no.~3, 1442--1450.

\bibitem{AMA.20}
A. Debbouche, D. Baleanu\ and\ R. P. Agarwal,
Nonlocal nonlinear integrodifferential equations of fractional orders,
Bound. Value Probl. {\bf 2012} (2012), no.~78, 10~pp.

\bibitem{AMA.21}
A. Debbouche\ and\ D. F. M. Torres,
Approximate controllability of fractional nonlocal delay semilinear systems in Hilbert spaces,
Internat. J. Control {\bf 86} (2013), no.~9, 1577--1585.
{\tt arXiv:1304.0082}

\bibitem{AMA.22}
G. M. N'Gu\'er\'ekata,
A Cauchy problem for some fractional abstract differential equation
with non local conditions,
Nonlinear Anal. {\bf 70} (2009), no.~5, 1873--1876.

\bibitem{AMA.23}
J. R. Wang, Y. Zhou, W. Wei\ and\ H. Xu,
Nonlocal problems for fractional integrodifferential equations
via fractional operators and optimal controls,
Comput. Math. Appl. {\bf 62} (2011), no.~3, 1427--1441.

\bibitem{AMA.24}
D. Mozyrska\ and\ D. F. M. Torres,
Minimal modified energy control for fractional linear
control systems with the Caputo derivative,
Carpathian J. Math. {\bf 26} (2010), no.~2, 210--221.
{\tt arXiv:1004.3113}

\bibitem{AMA.25}
D. Mozyrska\ and\ D. F. M. Torres,
Modified optimal energy and initial memory of fractional
continuous-time linear systems,
Signal Process. {\bf 91} (2011), no.~3, 379--385.
{\tt arXiv:1007.3946}

\bibitem{AMA.26}
X. Zhao, N. Duan\ and\ B. Liu,
Optimal control problem of a generalized Ginzburg–Landau model equation in population problems,
Math. Meth. Appl. Sci. {\bf 37} (2014), no.~3, 435--446.

\bibitem{AMA.27}
G. F. Franklin, J. D. Powell\ and\ A. Emami-Naeini,
{\it Feedback control of dynamic systems},
Prentice Hall, 6th edition, 2010.

\bibitem{AMA.28}
N. \"Ozdemir, D. Karadeniz\ and\ B. B. \.Iskender,
Fractional optimal control problem of a distributed system in cylindrical coordinates,
Phys. Lett. A {\bf 373} (2009), no.~2, 221--226.

\bibitem{AMA.29}
J. Wang\ and\ Y. Zhou,
A class of fractional evolution equations and optimal controls,
Nonlinear Anal. Real World Appl. {\bf 12} (2011), no.~1, 262--272.

\bibitem{MR3181063}
M. Jelassi\ and\ H. Mejjaoli,
Fractional Sobolev type spaces associated with a singular differential operator and applications,
Fract. Calc. Appl. Anal. {\bf 17} (2014), no.~2, 401--423.

\bibitem{AMA.30}
M. Fe\u ckan, J. Wang\ and\ Y. Zhou,
Controllability of fractional functional evolution equations
of Sobolev type via characteristic solution operators,
J. Optim. Theory Appl. {\bf 156} (2013), no.~1, 79--95.

\bibitem{MR3154621}
A. Harrat\ and\ A. Debbouche,
Sobolev type fractional delay impulsive equations with Alpha-Sobolev resolvent families and integral conditions,
Nonlinear Stud. {\bf 20} (2013), no.~4, 549--558.

\bibitem{MR3129325}
M. Kerboua, A. Debbouche\ and\ D. Baleanu,
Approximate controllability of Sobolev type nonlocal fractional stochastic dynamic systems in Hilbert spaces,
Abstr. Appl. Anal. {\bf 2013}, Art. ID 262191, 10~pp.

\bibitem{AMA.31}
F. Li, J. Liang\ and\ H. K. Xu,
Existence of mild solutions for fractional integrodifferential
equations of Sobolev type with nonlocal conditions,
J. Math. Anal. Appl. {\bf 391} (2012), no.~2, 510--525.

\bibitem{AMA.32}
E. Hille\ and\ R. S. Phillips,
{\it Functional analysis and semi-groups},
American Mathematical Society Colloquium Publications,
vol. 31, Amer. Math. Soc., Providence, RI, 1957.

\bibitem{AMA.33}
A. Pazy,
{\it Semigroups of linear operators and applications to partial differential equations},
Applied Mathematical Sciences, 44, Springer, New York, 1983.

\bibitem{AMA.34}
S. D. Zaidman,
{\it Abstract differential equations},
Research Notes in Mathematics, 36, Pitman, Boston, MA, 1979.

\bibitem{AMA.35}
H. Liu\ and\ J.-C. Chang,
Existence for a class of partial differential equations with nonlocal conditions,
Nonlinear Anal. {\bf 70} (2009), no.~9, 3076--3083.

\bibitem{AMA.36}
Y. Zhou\ and\ F. Jiao,
Nonlocal Cauchy problem for fractional evolution equations,
Nonlinear Anal. Real World Appl. {\bf 11} (2010), no.~5, 4465--4475.

\bibitem{AMA.37}
Y. Zhou\ and\ F. Jiao,
Existence of mild solutions for fractional neutral evolution equations,
Comput. Math. Appl. {\bf 59} (2010), no.~3, 1063--1077.

\bibitem{AMA.38}
E. Zeidler,
{\it Nonlinear functional analysis and its applications. II/A},
translated from the German by the author
and Leo F. Boron, Springer, New York, 1990.

\bibitem{AMA.39}
S. Hu\ and\ N. S. Papageorgiou,
{\it Handbook of multivalued analysis. Vol. I},
Mathematics and its Applications, 419,
Kluwer Acad. Publ., Dordrecht, 1997.

\end{thebibliography}
\end{document}